\documentclass[reqno, a4paper, 10pt]{amsart}
\usepackage{amssymb}
\usepackage{mathrsfs}
\usepackage{amssymb, url, color, pb-diagram, graphicx, amscd, mathrsfs}
\usepackage[colorlinks=true, bookmarks=true, pdfstartview=FitH, pagebackref=true, linktocpage=true, linkcolor = magenta, citecolor = blue]{hyperref}
\usepackage[nodayofweek]{datetime}
\usepackage{graphicx}
\usepackage{times}

\usepackage[pagewise,running,mathlines]{lineno}
\usepackage[subnum]{cases}
\usepackage[compress]{cite}
\usepackage{enumitem}
\usepackage{empheq}

\DeclareMathAlphabet{\mathpzc}{OT1}{pzc}{m}{it}

\newtheorem{theorem}{Theorem}[section]
\newtheorem{lemma}[theorem]{Lemma}

\newtheorem{proposition}[theorem]{Proposition}

\theoremstyle{definition}

\theoremstyle{remark}
\newtheorem{remark}[theorem]{Remark}

\numberwithin{equation}{section}

\def\Xint#1{\mathchoice
 {\XXint\displaystyle\textstyle{#1}}%
 {\XXint\textstyle\scriptstyle{#1}}%
 {\XXint\scriptstyle\scriptscriptstyle{#1}}%
 {\XXint\scriptscriptstyle\scriptscriptstyle{#1}}%
 \!\int}
\def\XXint#1#2#3{{\setbox0=\hbox{$#1{#2#3}{\int}$}
 \vcenter{\hbox{$#2#3$}}\kern-.5\wd0}}

\def\dashint{\Xint-}

\newcommand{\Cscr}{\mathscr C}
\newcommand{\Po}{\mathbf{P}_{g_0}}
\newcommand{\PP}{\mathbf{P}}
\newcommand*{\doverline}[1]{\overline{\overline{#1}}}

\DeclareMathOperator{\Hess}{Hess}
\DeclareMathOperator{\R}{scal}

\DeclareMathOperator{\Ric}{Ric}

\DeclareMathOperator{\vol}{vol}
\DeclareMathOperator{\supp}{supp}
\DeclareMathOperator{\dvg}{d\mu_{{\it g}_0}}
\DeclareMathOperator{\dv}{d\mu}

\DeclareMathOperator{\dz}{d\mathit{z}}
\newcommand{\definedas}{\mathrel{\raise.095ex\hbox{\rm :}\mkern-5.2mu=}}

\let\vol\Vol
\let\epsilon\varepsilon

\allowdisplaybreaks

\begin{document}

\title[Bubbling of the prescribed $Q$-curvature equation]{Bubbling of the prescribed $Q$-curvature equation on 4-manifolds in the null case}

\author[Q.A. Ng\^{o}]{Qu\^{o}\hspace{-0.5ex}\llap{\raise 1ex\hbox{\'{}}}\hspace{0.5ex}c Anh Ng\^{o}}
\address[Q.A. Ng\^{o}]{Department of Mathematics\\ College of Science \\ Vi\^{e}t Nam National University\\ Ha N\^{o}i \\ Vi\^{e}t Nam}
\email{\href{mailto: Q. A. Ngo <nqanh@vnu.edu.vn>}{nqanh@vnu.edu.vn}}
\email{\href{mailto: Q. A. Ngo <bookworm\_vn@yahoo.com>}{bookworm\_vn@yahoo.com}}

\author[H. Zhang]{Hong Zhang}
\address[H. Zhang]{School of Mathematics, University of Science and Technology of China, No.96 Jinzhai Road, Hefei, Anhui, China, 230026.}
\email{\href{mailto: H. Zhang <matzhang@ustc.edu.cn>}{matzhang@ustc.edu.cn}}
\email{\href{mailto: H. Zhang <mathongzhang@gmail.com>}{mathongzhang@gmail.com}}
\thanks{}



\date{\bf \today \ at \currenttime}


\setpagewiselinenumbers
\setlength\linenumbersep{110pt}

\begin{abstract}
Analog to the classical result of Kazdan--Warner for the existence of solutions to the prescribed Gaussian curvature equation on compact 2-manifolds without boundary, it is widely known that if $(M,g_0)$ is a closed 4-manifold with zero $Q$-curvature and if $f$ is any non-constant, smooth, sign-changing function with $\int_M f \dvg <0$, then there exists at least one solution $u$ to the prescribed $Q$-curvature equation 
\[
\Po u = f e^{4u},
\] 
where $\Po$ is the Paneitz operator which is positive with kernel consisting of constant functions. In this paper, we fix a non-constant smooth function $f_0$ with
\[
\max_{x\in M}f_0(x)=0, \quad \int_M f_0 \dvg <0
\]
and consider a family of prescribed $Q$-curvature equations
\[
\Po u=(f_0+\lambda)e^{4u},
\]
where $\lambda>0$ is a suitably small constant. A solution to the equation above can be obtained from a minimizer $u_\lambda$ of certain energy functional associated to the equation. Firstly, we prove that the minimizer $u_\lambda$ exhibits bubbling phenomenon in a certain limit regime as $\lambda \searrow 0$. Then, we show that the analogous phenomenon occurs in the context of $Q$-curvature flow. 
\end{abstract}

\maketitle

\section{Introduction}

The problem of describing the set of curvatures that a given manifold can possess is of importance in Riemannian geometry over the last 50 years starting from a seminal paper in 1960, or even before, due to Yamabe \cite{Yamabe} for the existence of conformal metrics of constant scalar curvature on closed manifolds of dimension $n \geqslant 3$. Without limiting to the case of constant scalar curvature, this problem is known as the prescribed scalar curvature problem and has been a main research topic in conformal geometry in recent decades. An analogue problem for manifolds of dimension $2$, known as the prescribed Gaussian curvature problem, can be formulated in a similar way.

\subsection{The Kazdan--Warner result for the scalar curvature equation}

Let $(M, g_0)$ be a compact surface without boundary. Given a smooth function $f$ on $M$, the prescribed Gaussian curvature problem asks if there exists a conformal metric $g$ such that the Gaussian curvature of $g$ is equal to $f$. By writing $g=e^{2u}g_0$, the Gaussian curvature of the metric $g$, denoted by $K_g$, satisfies the transformation law
$$K_g=e^{-2u}(-\Delta_{g_0}u+K_{g_0}).$$
This enables us to reduce the prescribed Gaussian curvature problem to the problem of solving the semilinear PDE 
\begin{equation}\label{eqPGaussianP}
-\Delta_{g_0} u + K_{g_0} = f e^{2u}
\end{equation}
 Since Eq.\eqref{eqPGaussianP} is conformally covariant, we obtain that if $v$ solves
\[
-\Delta_{g_1} v + K_{g_1} = f e^{2 v}
\]
for some $g_1 = e^{2w}g_0$, then $u = v + w$ solves \eqref{eqPGaussianP}. This together with the uniformization theorem implies that we can freely choose the background metric $g_0$ in such a way that $K_{g_0}$ is a constant whose sign depends on the Euler characteristic of $M$. In the case that $M$ has genus one, namely, $M$ is the torus, Eq. \eqref{eqPGaussianP} becomes
\begin{equation}\label{eqKW}
-\Delta_{g_0} u = fe^{2u}
\end{equation}
on $M$.
In \cite{KW}, Kazdan and Warner proved the following result:

\begin{theorem}[see Kazdan--Warner \cite{KW}]\label{KW}
There is a solution $u$ to \eqref{eqKW} if, and only if, either $f \equiv 0$, or if the function $f$ changes sign and satisfies
\begin{equation}\label{eqKWTotalIntegralIsNegative}
\int_M f \dvg < 0.
\end{equation}
\end{theorem}

A solution $u$ to \eqref{eqKW} can be obtained by minimizing the Liouville energy 
$$E(u)=\frac12\int_M |\nabla u|^2 \dvg$$
in the class 
\[
C_f = \Big\{ u \in H^1(M, g_0) : \int_M f e^{2u} \dvg = 0 \Big\}.
\]
We note that the constraint $\int_M f e^{2u} \dvg = 0$ in the class $C_f$ is quite natural in view of the Gauss--Bonnet theorem. Since the energy $E$ and the constraint in $C_f$ is left unchanged up to a constant addition, in order to show existence of a minimizer for $E$ in the class $C_f$, one often restricts attention to those functions with vanishing mean. To be precise, we look for minimizer of $E$ within the set
\[
 C_f^\prime=\Big\{ u \in H^1(M, g_0) : \int_M f e^{2u} \dvg = 0, \, \int_M u \dvg = 0 \Big\}.
\] 
However, normalizing the volume will work equally well, that is, we can also look for minimizer of $E$ within the set
\[
 C_f^*=\Big\{ u \in H^1(M, g_0) : \int_M f e^{2u} \dvg = 0, \, \int_M e^{2u} \dvg = \vol (M, g_0) \Big\}.
\] 
In \cite{Ga2015}, Galimberti showed ``bubbling'' of the Kazdan--Warner metrics in a certain limit regime. To describe his result, we let $f_0$ be a non-constant and smooth function with $\max_M f_0=0$. Let $\lambda >0$ be small such that $f_\lambda = f_0 + \lambda$ changes sign and satisfies \eqref{eqKWTotalIntegralIsNegative}. Therefore, by Theorem \ref{KW} there exists a solution $\widehat u_\lambda$ to \eqref{eqKW}, which can be obtained from a minizer $u_\lambda$ of $E$ in the set $C_{f_\lambda}^*$ with $f$ replaced by $f_\lambda$. In fact, one can easily see that $\widehat u_\lambda$ and $u_\lambda$ differ by a positive constant $c_\lambda$. With a delicate argument, he is able to control the total curvature of the conformal metrics $\widehat g_\lambda = e^{2\widehat u_\lambda} g_0$ for suibtable $\lambda \searrow 0$ and hence to show that after rescaling the metrics suitably near local maximum points of $f$, one or more ``bubbles'' may be extracted from $\widehat g_\lambda$; see \cite[Theorem 1.1]{Ga2015}.

Recently, Struwe \cite{Str} improves the result in \cite{Ga2015} by obtaining a more precise characterization of the bubbling. He shows that ``slow blow-up'' does not occur; see \cite[Theorem 1.2]{Str}. This is achieved with the help of a new Liouville-type result; see \cite[Theorem 1.3]{Str}. It is remarkable that the method developed in \cite{Str} is flexible enough to apply also in the presence of perturbation leading to a similar ``bubbling'' phenomenon for a family of prescribed curvature flows for $f_\lambda$ with suitably chosen initial data in $C_{f_\lambda}$; see \cite[Theorem 1.5]{Str}.

In the last paragraph of Subsection 1.5 in \cite{Str}, Struwe comments on future investigation of ``bubbling'' metrics of prescribing $Q$-curvature equation in arbitrary even dimensions $n \geqslant 4$. Inspired by his interesting work and comments, we aim to study the bubbling behavior of the prescribed $Q$-curvature equation in the null case. In fact, we have borrowed many ideas from \cite{Str} in the proof of the main theorems in the paper. 

\subsection{A Kazdan--Warner type result for the $Q$-curvature equation}

Let $(M,g_0)$ be a closed 4-dimensional Riemannian manifold endowed with a smooth background metric $g_0$. An analogue of the conformal Laplacian in dimension $2$ is the Paneitz operator $ \Po $ discovered by \cite{Pan}. To be more precise, it is defined in terms of the Ricci tensor $\Ric_{g_0}$ and the scalar curvature $\R_{g_0}$ as 
\[
 \Po =\Delta_{g_0}^2-\mbox{div}_{g_0}\bigg((\frac23 \R_{g_0}g_0-2 \Ric_{g_0})d\bigg).
\]
Associated to the Paneitz operator $ \Po $, Branson \cite{Bra} found the $Q$-curvature which enjoys many similar properties as the Gaussian curvatue in dimension 2. It is also given, in terms of the Ricci tensor $\Ric_{g_0}$ and the scalar curvature $R_{g_0}$, by
$$Q_{g_0}=-\frac16\Big(\Delta_{g_0}\R_{g_0}-R^2_{g_0}+3| \Ric_{g_0}|^2\Big).$$
An important topic about the $Q$-curvature is the prescribed $Q$-curvature problem which is formulated as follows. Given a smooth function $f$ on $M$, one may ask if there exists a conformal metric $g=e^{2u}g_0$ with $Q$-curvature $Q_g=f$. To solve the geometric problem is equivalent to finding the solution to the fourth order semilinear PDE.
\begin{equation}\label{generalpQe}
 \Po u+Q_{g_0}=fe^{4u}.
\end{equation}
 There are many research works on the equation \eqref{generalpQe}, see, for instance, \cite{BFR, Bre, CY, DM, LLL, MS, WX} and references therein. 

In this paper, we consider the prescribed $Q$-curvature equation on 4-manifolds in the null case, that is, $\int_MQ_{g_0} \dvg =0$. Due to the resolution of the constant $Q$-curvature problem, we may assume, w.l.o.g., that the background metric $g_0$ has the null $Q$-curvature. Then the equation \eqref{generalpQe} becomes
\begin{equation}
 \label{pQe}
 \Po u=fe^{4u}.
\end{equation}
If $f\not\equiv0$, then it is necessary that $f$ changes sign for the existence of a solution to \eqref{pQe}, since $\int_Mfe^{4u} \dvg =0$. However, unlike the two-dimensional case, $\int_M f \dvg <0$ is not necessary anymore. The following result shows that $\int_Mf \dvg <0$ is still sufficient.

\begin{theorem}[see Ge-Xu \cite{GX}]\label{GeXu}
 Let $(M, g)$ be a compact, oriented four-dimensional Riemannian manifold. Assume that the Paneitz operator $ \Po 
$ is positive with kernel consisting of constant functions. If 
\[
\sup_Mf>0~~\mbox{ and }~~\int_{M}f \dvg <0,
\]
then there exists a smooth solution to \eqref{pQe}.
\end{theorem}

In \cite{GX}, Ge and Xu proved that a solution to \eqref{pQe} may be obtained by minimizing the energy
\[
\mathscr{E}(u)=2\langle \Po u,u\rangle
\]
under the constraint
\[
F=\bigg\{u\in H^2(M,g_0):\int_Mfe^{4u} \dvg =0~~\mbox{and}~~\int_Mu \dvg =0\bigg\}.
\]
Here, for $u, v\in H^2(M,g_0)$, the inner product $\langle \Po u,v\rangle$ is defined as follows
$$\langle \Po u,v\rangle=\int_M\Big[\Delta_{g_0}u\Delta_{g_0}v+\frac23R_{g_0}g_0(\nabla_{g_0}u,\nabla_{g_0}v)-2Ric_{g_0}(\nabla_{g_0}u,\nabla_{g_0}v)\Big] \dvg .$$
However, similar to the case of \eqref{eqKW}, the authors showed, in \cite[Theorem A.1]{NZ}, that the way of searching a solution is still successful if we minimize $\mathscr{E}(u)$ under the following constraint
\[
X_f^*=\bigg\{u\in H^2(M,g_0):\int_Mfe^{4u} \dvg =0~~\mbox{ and }~~\int_Me^{4u} \dvg =1\bigg\}.
\]

\section{Main results}

We shall study ``bubbling'' of the prescribed $Q$-curvature equation on $4$-manifolds in two different contexts: the static case and the flow case.

\subsection{Bubbling metrics in the static case}

As in \cite{Ga,Str} , we let $f_0$ be a smooth, non-constant function with $\max_{x\in M}f(x)=0$, and let $f_\lambda=f_0+\lambda$ for any $\lambda\in\mathbb{R}$. By assuming that $ \vol (M,g_0)=1$, we find that if 
\begin{equation}\label{rangeoflambda}
0<\lambda<-\int_Mf_0 \dvg :=\lambda_0,
\end{equation}
then $f_\lambda$ changes sign and $\int_Mf_\lambda \dvg <0$. Hence, it follows from Theorem \ref{GeXu} that there exists a solution $\widetilde{u}_\lambda$ to \eqref{pQe} with $f$ replaced by $f_\lambda$. In addition, \cite[Theorem A.1]{NZ} implies that $\widetilde{u}_\lambda$ can be obtained as 
$$\widetilde{u}_\lambda=u_\lambda+c_\lambda$$ 
from a minimizer $u_\lambda$ of $\mathscr{E}$ in the set $ X_{f_\lambda}^*$. Here $u_\lambda$ satisfies
\begin{equation}\label{pQemini}
 \Po u_\lambda=\alpha_\lambda f_\lambda e^{4u_\lambda},
\end{equation}
with $\alpha_\lambda>0$ and $c_\lambda=(\log\alpha_\lambda)/4$. Moreover, by setting 
\[
\widetilde{g}_{\lambda}=e^{2\widetilde{u}_\lambda}g_0,
\] 
we have
\begin{equation}\label{alphalambda}
\alpha_\lambda=e^{4c_\lambda}=\int_Me^{4(u_\lambda+c_\lambda)} \dvg =\vol(M,\widetilde{g}_\lambda).
\end{equation}
Also, set
\begin{equation}\label{betalambda}
\beta_\lambda:=\mathscr{E}(u_\lambda)=\min\big\{\mathscr{E}(u):u\in X^*_{f_\lambda}\big\}.
\end{equation}
Then one will see from Lemma \ref{unbdofbeta} below that $\beta_\lambda \to +\infty$ as $\lambda\searrow0$; Thus, one should expect the bubbling phenomenon associated with the family of metrics $\widetilde{g}_\lambda$ to occur. 

The purpose of this part of the paper is to characterize the bubbling behavior of $\widetilde{g}_\lambda$. First, when the function $f_0$ has only non-degenerate maxima, we have the following result:

\begin{theorem}\label{main1}
Assume that the Paneitz operator $ \Po 
$ is positive with kernel consisting of constant functions. Let $f_0\leqslant0$ be a smooth, non-constant function with $\max_Mf_0=0$ having only non-degenerate maximum points. Then for suitable $\lambda_k\searrow0$, for $u_k=u_{\lambda_k}$ as above and suitable $I\in\mathbb{N}$, $r_k^{(i)}\searrow0$, $x_k^{(i)} \to x_\infty^{(i)}\in M$ with $f_0(x_\infty^{(i)})=0$, $i\leqslant i\leqslant I$, as $k \to +\infty$ the following hold:
\begin{enumerate}[label=\rm (\roman*)]
 \item $u_k \to -\infty$ locally uniformly on 
 $M_\infty=M\backslash\{x_\infty^{(i)}:1\leqslant i\leqslant I\}$.
 \item In normal coordinates around $x_\infty^{(i)}$, set
\[
z_k^{(i)}=\exp_{x_\infty^{(i)}}^{-1}(x_k^{(i)}), \quad \widetilde{u}_k=u_k\circ\exp_{x_\infty^{(i)}}.
\] 
Then for each $1\leqslant i\leqslant I$, either
\begin{enumerate}[label=\rm (\alph*)]
 \item $\limsup_{k\to\infty}r_k^{(i)}/\sqrt{\lambda_k}=0$ and
\begin{align*}
\widehat{u}_k(z):=
\widetilde u_k \big(z_k^{(i)}+r_k^{(i)}z\big)+&\log r_k^{(i)}
\to \widehat{u}_\infty(z)
\end{align*} 
 strongly in $H_{\rm loc}^4(\mathbf R^4)$, where $\widehat u_\infty$, up to a translation and a scaling, is given by 
 $$\widehat u_\infty(z)=\log\bigg( \frac{4\sqrt{6}}{4\sqrt{6} + |z|^2 }\bigg)$$
 and it induces a spherical metric
\[
\widehat{g}_\infty=e^{4\widehat{u}_\infty}g_{\mathbf R^4}
\]
of $Q$-curvature
\[
Q_{\widehat{g}_\infty}\equiv1
\]
on $\mathbf R^4$ and $1\leqslant I\leqslant 4$, or

\item $\limsup_{k\to\infty}r_k^{(i)}/\sqrt{\lambda_k}>0$ and
\begin{align*}
\widehat{u}_k(z):=
\widetilde u_k\big(z_k^{(i)}+r_k^{(i)}z\big)+\log r_k^{(i)} \to \widehat{u}_\infty(z)
\end{align*}
 strongly in $H^4_{\rm loc}(\mathbf R^4)$, where $\widehat u_\infty$, up to a translation and a scaling, solves
 \begin{equation}\label{eqLimitingEquationW-slow0}
 \Delta_z^2\widehat u_\infty(z)=\Big(1+\frac12{\rm Hess}_{f_0}\big(x_\infty^{(i)}\big)\big[z,z\big]\Big)e^{4\widehat u_\infty(z)},
 \end{equation}
In addition, the metric
\[
\widehat{g}_\infty=e^{4\widehat{u}_\infty}g_{\mathbf R^4}
\]
on $\mathbf R^4$ has finite volume and finite total $Q$-curvature
\[
Q_{\widehat{g}_\infty}(z)=1+\frac12{\rm Hess}_{f_0}\big(x_\infty^{(i)}\big)\big[z,z\big]
\]
and $1\leqslant I\leqslant 8$.
\end{enumerate}
\end{enumerate}
\end{theorem}

\begin{remark}
Unlike the Struwe's result in \cite{Str}, the ``slow blow-up'' case (b) is unable to be ruled out here. In fact, the limiting equation \eqref{eqLimitingEquationW-slow0} associated with blow-up points $x_\infty^{(i)}$ with $1 \leqslant i \leqslant I$ may have a solution with finite energy and finite total curvature. To see this, one may apply a general existence result due to Chang and Chen \cite{ChangChen01} to obtain that there is a solution to
\[
\Delta^2_z\widehat u_\infty=\Big(1+\frac12 \Hess_{f_0}\big(x_\infty^{(i)}\big)\big[z,z\big]\Big)e^{4\widehat u_\infty}
\]
with
\[
\int_{\mathbf R^4} e^{4\widehat u_\infty(z)} \dz <+\infty.
\] 
and
\[
\int_{\mathbf R^4}\Big(1+\frac12\Hess_{f_0}\big(x_\infty^{(i)}\big)[z,z]\Big)e^{4\widehat u_\infty(z)} \dz <+\infty.
\]
Since $x_\infty^{(i)}$ is a non-degenerate maxima of $f_0$, the matrix $\Hess_{f_0}(x_\infty^{(i)})$ is negative definite. Consequently, we also have 
\[
\int_{\mathbf R^4}\Big|1+\frac12\Hess_{f_0}\big(x_\infty^{(i)}\big)[z,z]\Big|e^{4\widehat u_\infty(z)} \dz <+\infty.
\]
\end{remark}

Now, we consider the case that the function $f_0$ may have a degenerate maxima. To describe our next result, motivated by \cite{Str}, we propose the following condition on $f_0$ analog to Condition A in \cite{Str}. 

\noindent{\bf Condition A}: Let $M_0=\{x\in M: f_0(x)=0\}$ and $d(x)={\rm dist}(x,M_0)$ for $x \in M$. There exist $d_0>0$ and $A_0>0$ such that, letting 
\[
K_0=\Big\{z=(z^1,z^2,z^3,z^4)\in\mathbf R^4 :\sqrt{\sum_{i=1}^3(z^i)^2}<z^4, |z|<d_0\Big\}
\]
for any $x\in M$ with $0<d(x)<d_0$ there is a rotated copy $K_x\subset\mathbf R^4$ of $K_0$ with vertex at $x$ such that in Euclidean coordinates $z$ around $x=0$ there holds
\[
A_0\inf_{z\in K_x}|f_0(\exp_x(z))|\geqslant|f_0(x)|.
\]
Since any function on a closed manifold with only non-degenerate maxima admits finitely many maximum points, it is then clear to see that Condition A is automatically satisfied by such functions. Let us take one example of a function $f_0$ satisfying Condition A. We use $(r, \theta_1, \theta_2, \theta_3)$ to denote the polar coordinates in the Euclidean space $\mathbf R^4$. Let $f_0$ be as follows
\[
f_0 (r, \theta_1, \theta_2, \theta_3)= 
\begin{cases}
0 & \text{ if } r \leqslant 1,\\
\displaystyle -e^{-1/(r-1)} \Big( \sum_{i=1}^3 \sin \big( \frac 1{r-1} + \theta_i \big) + 4 \Big) & \text{ if } r > 1.
\end{cases}
\]
Then it is straightforward to verify that the function $f_0$ above satisfies Condition A with $A_0=7$. Furthermore, $f_0$ has degenerate maximum points.

Return to characterizing the bubbling behavior of $\widetilde{g}_\lambda$ in the degenerate situation, our second result reads as follows.

\begin{theorem}\label{main}
 Assume all the conditions, expcept for the assumption of the non-degeneracy of the function $f_0$ at a maxima, in Theorem \ref{main1} above. If, in addition, $(M,g_0)$ is locally conformally flat and $f_0$ satisfies the {\bf Condition A} with $d_0, A_0>0$, then for $u_k$ defined as in the Theorem \ref{main1} there exist suitable $I\in\mathbb{N}$ with $I \leqslant 8$, $r_k^{(i)}\searrow0$ and $x_k^{(i)} \to x_\infty^{(i)}\in M$ with $f_0(x_\infty^{(i)})=0$, $1\leqslant i\leqslant I$ such that the following hold
\begin{enumerate}[label=\rm (\roman*)]
 \item $u_k \to -\infty$ locally uniformly on $M_\infty=M\backslash\{x_\infty^{(i)}:1\leqslant i\leqslant I\}$.
 \item For each $1\leqslant i\leqslant I$, we have
\[
\widehat{u}_k(z):=
 \widetilde{u}_k\big(z_k^{(i)}+r_k^{(i)} z\big)+\log r_k^{(i)} \to \widehat{u}_\infty(z)
\]
 strongly in $H_{\rm loc}^4(\mathbf R^4)$, where $z_k^{(i)}=\exp_{x_\infty^{(i)}}^{-1}(x_k^{(i)})$ and $\widehat{u}_\infty$ induces a metric
\[
\widehat g_\infty=e^{4\widehat{u}_\infty}g_{\mathbf R^4}
\]
on $\mathbf R^4$ of locally bounded curvature and of volume less than or equal $1$.
\end{enumerate}
\end{theorem}
\begin{remark}
By comparing Theorems \ref{main1} and \ref{main}, one can easily notice that in the degenerate case we made an extra assumption on the manifold $(M,g_0)$ except for the Condition A, that is, we require the manifold $(M,g_0)$ to be locally conformally flat. It would be interesting to investigate the bubbling phenomenon in the degenerate situation without assuming the locally conformal flatness.
\end{remark}

\subsection{Bubbling metrics along the prescribed curvature flow}

In contrast to the statics case, our second goal is to obtain an analogous bubbling behavior described in Theorem \ref{main1} for a family of prescribed $Q$-curvature flows for $f_\lambda$ with suitably chosen initial data in $X_{f_\lambda}$, where
\[
X_{f_\lambda}=\Big\{u\in H^2(M): \int_Mf_\lambda e^{4u} \dvg =0\Big\}.
\]
To describe our second result precisely, let us briefly recall the prescribed $Q$-curvature flow introduced in \cite{NZ}. Let $g_\lambda(t)=e^{2u_\lambda(t)}g_0$ be a family of time-dependent conformal metrics satisfying
$$\frac{\partial g_\lambda}{\partial t}=-2(Q_{g_\lambda}-\alpha_\lambda(t) f_\lambda)g_\lambda$$
with the initial conformal metric $g_\lambda(0)=e^{2u_{0\lambda}}g_0$. In terms of $u_\lambda(t)$, the evolution equation above becomes
\begin{equation}
 \label{eeforu}
 \frac{\partial u_\lambda}{\partial t}=\alpha_\lambda(t) f_\lambda-Q_{g_\lambda}
\end{equation}
with the initial data
\[
u_\lambda(0)=u_{0\lambda}\in X_{f_\lambda}.
\] 
The function $\alpha_\lambda=\alpha_\lambda(t)$ is chosen in such a way that $\int_Mf_\lambda \dv_{g_\lambda}$ remains constant, namely,
\begin{equation}\label{conditionalpha}
 \frac{d}{dt}\int_Mf_\lambda \dv_{g_\lambda}=4\int_M{u_\lambda}_tf_\lambda \dv_{g_\lambda}=4\int_M(\alpha_\lambda f_\lambda-Q_{g_\lambda})f_\lambda \dv_{g_\lambda}=0.
\end{equation}
Solving \eqref{conditionalpha} for $\alpha_\lambda$ gives
\begin{equation*}
 \alpha_\lambda=\frac{\int_Mf_\lambda Q_{g_\lambda} \dv_{g_\lambda}}{\int_Mf_\lambda^2 \dv_{g_\lambda}}.
\end{equation*}
It is easy to verify that
\[
u_\lambda(t)\in X_{f_\lambda}
\]
for all $t\geqslant0$. We thus have by conformal invariant of $Q$-curvature that
$$\frac14\frac{d}{dt} \vol(M,g_\lambda(t))=\int_M{u_\lambda}_t \dv_{g_\lambda}=\alpha_\lambda\int_Mf_\lambda \dv_{g_\lambda}-\int_MQ_{g_\lambda} \dv_{g_\lambda}=0.$$
Normalizing the initial metric $g_\lambda(0)$ to satisfy $\vol(M,g_\lambda(0))=1$, we then get
\begin{equation}\label{volumekeeping}
 \vol (M,g_\lambda(t))=\int_M \dv_{g_\lambda}=\int_M \dv_{g_\lambda(0)}=1
\end{equation}
for all $t>0$. This implies that
\begin{equation}\label{FlowRemainInSpace}
u_\lambda(t)\in X^*_{f_\lambda}
\end{equation}
for all $t\geqslant0$.

By applying \cite[Theorem 1.1]{NZ} to $f_\lambda$, we obtain the sequential convergence of the flow \eqref{eeforu}.

\begin{theorem}[see Ng\^{o}--Zhang \cite{NZ}]\label{NgoZh}
 The flow \eqref{eeforu} has a smooth solution $u_\lambda(t)$ on $[0,+\infty)$. Moreover, there exists a suitable time sequence $(t_j)_j$ with $t_j \to +\infty$ as $j \to +\infty$ and a suitable non-zero constant $\alpha_{\infty\lambda}\in\mathbb{R}$ such that $u_\lambda(t_j) \to u_{\infty\lambda}$ in $C^\infty(M,g_0)$,
 $|\alpha_\lambda(t_j)-\alpha_{\infty\lambda}| \to 0$ and $ \| Q_{g_\lambda}(t_j)-\alpha_{\infty\lambda}f_\lambda \| _{C^\infty(M,g_0)} \to 0$ as $j \to +\infty$. Finally, $u_{\infty\lambda}$ satisfies
 $$ \Po u_{\infty\lambda}=\alpha_{\infty\lambda}f_{\lambda}e^{4u_{\infty\lambda}}.$$
\end{theorem}

For any $0<\lambda<\lambda_0$ and any $\sigma\in(-\sigma_0,0)$, with the number $\sigma_0=\sigma_0(\lambda)$ to be determined in Lemma \ref{bdoftotalQcurvature1} below, we choose $u_{0\lambda}^\sigma\in X_{f_\lambda}^*$ such that
\begin{equation*}
 \mathscr{E}(u_{0\lambda}^\sigma)\leqslant\beta_\lambda+\sigma^2.
\end{equation*}
For such an initial data $u_{0\lambda}^\sigma$, it follows from Theorem \ref{NgoZh} that the flow \eqref{eeforu} possesses the smooth solution $u_\lambda^\sigma=u_\lambda^\sigma(t)$ with $\alpha_\lambda^\sigma=\alpha_\lambda^\sigma(t)$. Unlike the case of prescribed Gaussian curvature flow in the dimension two, the sign of $\alpha_{\infty\lambda}$ in the $Q$-curvature flow is unable to be determined. So, we have to assume that there exist a sequence $(\lambda_k)_k, k\in\mathbb{N}$ with $\lambda_k\searrow0$ as $k \to +\infty$ such that $\alpha_{\infty\lambda_k}>0$ for all $k$ large. With $\sigma_k$ and $T_k$ defined by \eqref{sequentialbdoftotalQcurvature} below, we let, for a suitable time sequence $(t_k)_k$ with $t_k\geqslant T_k$,
\begin{equation}\label{sequentialflow}
u_k=u_{\lambda_k}^{\sigma_k}(t_k), \alpha_k=\alpha_{\lambda_k}^{\sigma_k}(t_k).
\end{equation}

Now, our second result reads as

\begin{theorem}
 \label{main2}
 Let $f_0$ be, respectively, as in the Theorems \ref{main1} and \ref{main} above. Then for $\lambda_k\searrow0$ with $\alpha_{\infty\lambda_k}>0$, suitable $u_{0\lambda_k}\in X^*_{f_{\lambda_k}}$ with $\mathscr{E}(u_{0\lambda_k})-\beta_{\lambda_k}\leqslant\sigma_k^2\searrow0$, and sufficiently large $t_k\geqslant T_k \to +\infty$ as $k \to +\infty$, the conclusions of Theorems \ref{main1} and \ref{main} hold for $u_k$ defined by \eqref{sequentialflow}.
\end{theorem}

Our paper is organized as the table of contents below.

\tableofcontents

\section{Notations and preliminaries}

In this brief section, we collect some useful facts frequently used throughout the paper. First, given a function $w$ on $M$, let us denote by $\overline w$ the average of $w$ over $(M,g_0)$, namely,
\[
\overline{w}=\int_M w \dvg.
\]
(Keep in mind that $\vol (M, g_0)=1$.) We shall use a double bar for $w$, namely $\doverline w$, if we want to emphasize that the average of $w$ is taking over $M$ with any other conformal metric.

Recalling that the higher order Moser--Trudinger inequality for Paneitz operator $\Po$, known as Adam's inequality; see \cite[Theorem 2]{adams} states that if $ \Po $ is self-adjoint and positive with kernel consisting of constant functions, then there is some constant $\Cscr_A>0$ such that
\begin{equation}\label{eqAdamsTypeInequality}
\int_M \exp\Big( 32\pi^2 \frac{(u-{\overline u})^2}{\langle \Po u, u \rangle } \Big)~ \dvg \leqslant \Cscr_A
\end{equation}
for every $u \in H^2(M, g_0)$. As a consequence of \eqref{eqAdamsTypeInequality} and Young's inequality, we obtain the following inequality
\begin{equation}\label{TrudingerInequality}
\begin{split}
\int_M \exp\big(\alpha(u-{\overline u}) \big) ~ \dvg \leqslant & \Cscr_A \exp \Big( \frac{\alpha^2}{128 \pi^2} \langle \Po u, u \rangle \Big)
\end{split}
\end{equation}
for all real number $\alpha$.

Now we collect some information of Green's function, denoted by $\mathbb G$, of the Paneitz operator $ \Po $. By the results in \cite{CY}, Green's function $\mathbb G$ is symmetric and fulfills the following properties:
\begin{enumerate}
 \item [(P1)] $\mathbb G$ is smooth on $M\times M\backslash\mbox{diagonal}$;
 \item [(P2)] there exists a positive constant $\Cscr_{\mathbb G}$ depending only on $(M,g_0)$ such that
\[
\Big|\mathbb G(x,y)-\frac{1}{8\pi^2}\log\frac{1}{d(x,y)}\Big|\leqslant \Cscr_{\mathbb G}
\]
for any $x, y\in M$ with $x\neq y$; while for its derivatives and for $1 \leqslant j \leqslant 3$ there holds
\[
\big|\nabla^j \mathbb G(x,y) \big|\leqslant \frac{\Cscr_{\mathbb G} }{d(x,y)^{j}}
\]
for any $x, y\in M$ with $x\neq y$.
\end{enumerate}
As clearly described in \cite[page 145]{Ma}, the higher order estimates in (P2) are not shown in \cite{CY} but they can be derived with the same approach, by an expansion of $\mathbb G$ at higher order using the parametrix. 

It is well known that if $\varphi\in L^1(M,g_0)$ with $\overline{\varphi}=0$, then $w$ solves
\[
 \Po w=\varphi,
\] 
if and only if
\begin{equation}\label{green}
 w(x)=\overline{w} +\int_M \mathbb G(x,y)\varphi(y) \dvg .
\end{equation}

For convenience, we cite the following lemma proved in \cite[Lemma 2.3]{Ma}. 

\begin{lemma}\label{Mal06Lemma2.3}
Let $(w_k)_k$ and $(\varphi_k)_k$ be two sequences of functions on $(M,g_0)$ satisfying
\[
\Po w_k=\varphi_k
\]
with $\|\varphi_k\|_{L^1(M,g_0)}\leqslant \alpha_0$ for some positive constant $\alpha_0$ independent of $k$. Then for any $x\in M$, any small $r>0$, and any $s\in[1,4/j)$ with $j=1, 2, 3$, there holds
\[
\int_{B_r(x)}|\nabla^j w_k|^s~\dvg\leqslant Cr^{4-js},
\]
where $C$, independent of $k$, is a positive constant depending only on $\alpha_0, M$, and $s$.
\end{lemma}

To end the section, we provide the following concentration-compactness result proved in \cite[Proposition 3.1]{Ma}.
 
\begin{proposition}\label{Mal06Proposition3.1}
Let $(w_k)_k$ and $(\varphi_k)_k$ be two sequences of functions on $(M,g_0)$ satisfying
\[
\Po w_k=\varphi_k
\]
with $\|\varphi_k\|_{L^1(M,g_0)}\leqslant \alpha_0$
for some positive constant $\alpha_0$ independent of $k$. Then, up to a subsequence, we have one of the following alternatives:
\begin{enumerate}
 \item [(i)] either there exist some constant $s>1$ and some positive constant $C$ independent of $k$ such that
\[
 \int_Me^{4s(w_k-\overline{w}_k)}~\dvg\leqslant C,
\]
 \item [(ii)] or there exist points $x_1, x_2, \dots, x_L \in M$ such that for any $r>0$ and any $i\in\{1, \dots, L\}$ one has
\[
\liminf_{k\rightarrow+\infty}\int_{B_r(x_i)}|\varphi_k|~\dvg\geqslant8\pi^2.
\]
\end{enumerate}
\end{proposition}

\begin{remark}\label{Prop3.1}
As clearly stated in \cite{Ma}, Proposition \ref{Mal06Proposition3.1} remains valid if the metric $g_0$ is replaced by a sequence of metrics $(g_k)_k$ which is uniformly bounded in $C^{N}(M,g_0)$ for any $N\in\mathbb{N}$. It also holds true if one replaces $M$ by any bounded open ball in $\mathbf{R}^4$, in which all the functions are compactly supported.
\end{remark}


\section{Bubbling in the static case}
\label{sec-ProofStatic}

In this section, we are going to prove the ``bubbling'' phenomena in the static case, namely, Theorems \ref{main1} and \ref{main}. 

\subsection{Bounds for total curvature} 

We derive, in this subsection, the bounds for the total $Q$-curvature. As an initial step, we show the unboundedness of the minimum energy $\beta_\lambda$ defined in \eqref{betalambda}. 

\begin{lemma}\label{unbdofbeta}
 As $\lambda\searrow0$, there holds
 $\beta_\lambda \to +\infty.$
\end{lemma}

\begin{proof}
Assume by contradiction that $\beta_\lambda\leqslant C_1$ for some constant $C_1$. Thanks to $u_\lambda \in X_{f_\lambda}^*$, we can use $\mathscr{E}(u_\lambda)$ to bound $\exp (-4 \overline u_\lambda )$ from above by applying Adams' inequality \eqref{TrudingerInequality} as follows
\[
e^{-4 \overline u_\lambda} = \int_M \exp\big(4 (u_\lambda- \overline u_\lambda) \big) ~ \dvg \leqslant \Cscr_A \exp \Big( \frac{\mathscr{E}(u_\lambda)}{16 \pi^2} \Big).
\]
Keep in mind that $|f_0|=-f_0=\lambda-f_\lambda$. From this together with H\"older's inequality we can estimate
 \begin{align*}
 0 &<\Big(\int_M|f_0| \dvg \Big)^2 \\
&\leqslant \int_M|f_0|e^{-4u_\lambda} \dvg \int_M|f_0|e^{4u_\lambda} \dvg \\
& \leqslant \|f_0 \|_{L^\infty (M, g_0)} e^{-4 \overline u_\lambda} \int_M e^{-4(u_\lambda - \overline u_\lambda)} \dvg \int_M(\lambda-f_\lambda)e^{4u_\lambda} \dvg \\
&\leqslant \lambda \|f_0 \|_{L^\infty (M, g_0)} \Cscr_A^2 \exp \Big( \frac{\beta_\lambda}{8\pi^2} \Big) ,
 \end{align*}
which is obviously a contradiction if $\lambda$ is sufficiently small.
\end{proof}

The following monotonicity property result is a key gradient for the uniform bound of the total $Q$-curvature of the metric $\widetilde{g}_\lambda=e^{4\widetilde{u}_\lambda}g_0$.

\begin{lemma}\label{monotonicitybeta}
The function $\lambda\mapsto\beta_\lambda$ is non-increasing in $\lambda$ for small $0<\lambda<\lambda_0$ and 
\[
\limsup_{\mu\searrow\lambda}\frac{\beta_\mu-\beta_\lambda}{\mu-\lambda}\leqslant-\alpha_\lambda,
\]
where $\lambda_0$ is given in \eqref{rangeoflambda}.
\end{lemma}

\begin{proof}
Fix $\lambda \in (0, \lambda_0)$. As always, let $u_\lambda\in X_{f_\lambda}^*$ be a minimizer of $\mathscr{E}$ as above, namely $\int_Mf_\lambda e^{4u_\lambda} \dvg = 0$ and $\int_M e^{4u_\lambda} \dvg = 1$. Then for small $\sigma\in\mathbb{R}$ we have, by Taylor's expansion, that
 \begin{align*}
 \int_Mf_\lambda e^{4(u_\lambda+\sigma f_\lambda)} \dvg =&\int_Mf_\lambda \Big[e^{4(u_\lambda+\sigma f_\lambda)}-e^{4u_\lambda}\Big] \dvg \\
 =&\int_Mf_\lambda e^{4u_\lambda}\Big[e^{4\sigma f_\lambda}-1\Big] \dvg \\
 =&4\sigma\int_Mf_\lambda^2 e^{4u_\lambda} \dvg +O(\sigma^2)
 \end{align*}
and
\begin{align*}
 \int_M e^{4(u_\lambda+\sigma f_\lambda)} \dvg =&1+\int_M e^{4u_\lambda}\Big[e^{4\sigma f_\lambda}-1\Big] \dvg \\
=&1+4\sigma\int_Mf_\lambda e^{4u_\lambda} \dvg +O(\sigma^2),\\
=&1+O(\sigma^2).
 \end{align*}
 So, for $0<|\sigma| \ll 1$ sufficiently small, if we let 
 $$\mu=\lambda-4\sigma\int_Mf_\lambda^2 e^{4u_\lambda} \dvg +O(\sigma^2),$$
 then we can find some constant $c$ such that
 \begin{equation}
 \label{elementinXmu}
 u_\lambda+\sigma f_\lambda+c\in X_{f_\mu}^*.
 \end{equation}
In particular, for $\sigma<0$ sufficiently close to zero, we have $\mu>\lambda$ and $\sigma=O(\mu-\lambda)$.
Notice that it follows from \eqref{pQemini} that
\begin{align*}
 \mathscr{E}(u_\lambda+\sigma f_\lambda+c) =& 2 \langle u_\lambda+\sigma f_\lambda+c, \Po u_\lambda+\sigma \Po f_\lambda \rangle \\
=&\mathscr{E}(u_\lambda)+4\sigma\int_M f_\lambda \Po u_\lambda \dvg + 2 \sigma^2 \langle f_\lambda , \Po f_\lambda \rangle \\
=&\mathscr{E}(u_\lambda)+4\sigma \alpha_\lambda \int_M f_\lambda^2e^{4u_\lambda} \dvg +O(\sigma^2).
\end{align*}
Now, by \eqref{elementinXmu}, we get that
\begin{align*}
 \beta_\mu \leqslant&\mathscr{E}(u_\lambda+\sigma f_\lambda+c)\\
\leqslant & \mathscr{E}(u_\lambda)+4\sigma\alpha_\lambda\int_M f_\lambda^2e^{4u_\lambda} \dvg +O(\sigma^2)\\
 =&\beta_\lambda-\alpha_\lambda(\mu-\lambda)+O((\mu-\lambda)^2)<\beta_\lambda,
\end{align*}
for $\sigma<0$ sufficiently close to zero. Hence the map $\lambda\mapsto\beta_\lambda$ is non-increasing and 
\[
\limsup_{\mu\searrow\lambda}\frac{\beta_\mu-\beta_\lambda}{\mu-\lambda}\leqslant-\alpha_\lambda
\]
as claimed.
\end{proof}

We can find the following bound on $\beta_\lambda$.

\begin{lemma}\label{bdofbeta}
 There holds
 $$\limsup_{\lambda\searrow0}\frac{\beta_\lambda}{\log(1/\lambda)}\leqslant64\pi^2.$$
\end{lemma}

\begin{proof}
Let $p_0\in M$ be such that $f_0(p_0)=0$ and assume that $\lambda\in(0,\lambda_0)$. By fixing a natural number $N\geqslant5$, we can find a smooth conformal metric $g_N=e^{2\varphi_N}g_0$ such that 
 \begin{equation}
 \label{conformalnormalmetric}
 \mbox{det}\big(g_N\big)=1+O(r^N)_{r \searrow 0},
 \end{equation} 
where $r=|x|$ and $x$ are $g_N$-normal coordinates around $p_0$ which is identified as $0$ in this new coordinate system. Now, letting
\[
A=\frac 1 2 {\rm Hess}_{f_0}(p_0).
\]
Since $p_0$ is an isolated maxima of $f_0$, for a suitable constant $L>0$ we have
\[
f_0(x)=(Ax,x)+O(|x|^3)\geqslant-\frac{\lambda}{2}
\]
on $B_{\sqrt{\lambda}/L}(0)$, and therefore $f_\lambda\geqslant\lambda/2$ on $B_{\sqrt{\lambda}/L}(0)$ for all $\lambda\in(0,\lambda_0)$. Fix a cut-off function $\tau\in C_c^\infty([0,\infty))$ with $0\leqslant\tau\leqslant1$ and
\[
 \tau(t)=
\begin{cases}
 1 & \text{ if } 0\leqslant t\leqslant1/2,\\
 0 & \text{ if } t\geqslant1.
 \end{cases}
\]
For any $A_0>1$, we can find a smooth function $\xi\in C^\infty([0,\infty))$ such that $1 \leqslant\xi\leqslant 2$, $\xi^\prime\geqslant 0$, $\sup_{t\geqslant0}\xi^\prime(t)\leqslant A_0$, and
\[
 \xi(t)=
 \begin{cases}
 t & \text{ if } 0\leqslant t\leqslant1,\\
 2 & \text{ if } t\geqslant2.
\end{cases}
\]
 Then we define
\[
z_\lambda(x)=
 \begin{cases}
 \log(1/\lambda) & \text{ if } |x|\leqslant\lambda,\\
 \frac12\log(1/\lambda)~\xi\Big(\frac{2\log|x|}{\log(\lambda)}\Big)\tau(|x|) &\text{ if } \lambda\leqslant|x|\leqslant1.
 \end{cases}
\]
 It is easy to see that $z_\lambda\in C^\infty(B_1(0))$ with ${\rm supp}(z_\lambda) \subset\overline{B_1(0)}$. Finally, we define for 
\[
w_\lambda(x)=
\begin{cases} 
z_\lambda\big(\frac{Lx}{\sqrt{\lambda}}\big)& \text{ if } x\in B_{\sqrt{\lambda}/L}(0),\\
0 & \text{ if } x\in M \setminus B_{\sqrt{\lambda}/L}(0).
\end{cases}
\]
Then, $w_\lambda\in C^\infty(M)$ with ${\rm supp}(w_\lambda) \subset\overline{B_{\sqrt{\lambda}/L}(0)}$. Consider the continuous function $\eta: [0, \infty) \mapsto\mathbb{R}$ defined by
\[
\eta(s)=\int_M f_\lambda e^{4sw_\lambda} \dvg.
\]
It follows from \eqref{rangeoflambda} that $\eta(0)<0$. On the other hand, by the definition of $w_\lambda$ and the fact that $f_\lambda\geqslant\lambda/2$ on $B_{\sqrt{\lambda}/L}(0)$, we conclude that
\begin{align*}
\eta(s)=& \int_Mf_\lambda \dvg + \int_{B_{\sqrt{\lambda}/L}(0)} f_\lambda \big( e^{4sw_\lambda} - 1 \big) \dvg\\
\geqslant &\int_Mf_\lambda \dvg +\frac\lambda2\int_{B_{\sqrt{\lambda}/L}(0)} \big( e^{4sw_\lambda}-1 \big)\dvg .
\end{align*}
This implies that $\eta(s) \to +\infty$ as $s \nearrow +\infty$. Hence, there exists some $s(\lambda)\in(0,+\infty)$ depending on $\lambda$ such that
\[
0=\eta(s)=\int_Mf_\lambda e^{4s(\lambda) w_\lambda} \dvg ,
\]
that is $s(\lambda) w_\lambda\in X_{f_\lambda}$. In addition, we may find a constant $c(\lambda)$ such that
\[
s(\lambda)w_\lambda+c(\lambda)\in X_{f_\lambda}^*.
\] 
Now, we provide a more precise estimate of $s(\lambda)$. Since $\vol(M,g_0)=1$, $\supp (w_\lambda) \subset\overline{B_{\sqrt{\lambda}/L}(0)}$, and $\dv_{g_N} = e^{4\varphi_N}\dvg$ we get that
\begin{align*}
0 =& \int_Mf_\lambda e^{4s(\lambda)w_\lambda} \dvg \\
=& \int_{B_{\sqrt{\lambda}/L}(0)}f_\lambda e^{4s(\lambda)w_\lambda} \dvg +\int_{M\backslash B_{\sqrt{\lambda}/L}(0)}f_\lambda \dvg \\
 \geqslant&\frac\lambda2\int_{B_{\sqrt{\lambda}/L}(0)} e^{4 [s(\lambda)w_\lambda-\varphi_N]} \dv_{g_N}- \| f_0 \| _\infty.
\end{align*}
By \eqref{conformalnormalmetric}, we have $ \dv_{g_N}=\sqrt{1+O(r^N)}dx$. Thus, for any $\epsilon\in(0,1)$, we can find $\lambda_\epsilon\in(0,\lambda_0)$, independent of $s$, such that for any $\lambda\in(0,\lambda_\epsilon)$
\begin{align*}
 \int_{B_{\sqrt{\lambda}/L}(0)} e^{4[s(\lambda) w_\lambda-\varphi_N]} \dv_{g_N} \geqslant&(\min_Me^{-4\varphi_N})\int_{B_{\sqrt{\lambda}/L}(0)} e^{4s(\lambda) w_\lambda}\sqrt{1+O(r^N)}~dx\\
 \geqslant&(\min_Me^{-4\varphi_N})(1-\epsilon)\int_{B_{\sqrt{\lambda}/L}(0)} e^{4s(\lambda) w_\lambda}~dx.
\end{align*}
It follows from the definition of $z_\lambda$ and after substituting $y = Lx/\sqrt \lambda$ that
\begin{align*}
\lambda\int_{B_{\sqrt{\lambda}/L}(0)} e^{4s(\lambda) w_\lambda}~dx=&\frac{\lambda^3}{L^4}\int_{B_1(0)} e^{4s(\lambda) z_\lambda(y)}~dy\\
\geqslant& \frac{\lambda^{3-4s(\lambda)}}{L^4}\int_{B_{\lambda^{5/4}}(0)}~dy=\frac{\pi^2\lambda^{8-4s(\lambda)}}{2L^4}.
\end{align*}
By combining all estimates above, we obtain
\[
\frac{1}{4L^4} (\min_Me^{-4\varphi_N}) (1-\epsilon)\pi^2~\lambda^{8-4s(\lambda)} \leqslant \| f_0 \| _\infty.
\]
Solving the preceding inequality for $s$ gives
\begin{equation}\label{bdofs}
0<s(\lambda) \leqslant2+\frac{1}{4\log(1/\lambda)} \log\Big(\frac{4L^4 \| f_0 \| _\infty\max_Me^{4\varphi_N}}{(1-\epsilon)\pi^2}\Big) 
:=2+O(1/\log(1/\lambda)).
\end{equation}
Next, following the proof of \cite[Lemma 3.6]{Ga}, we obtain that given any $A_0 > 1$, there exists $\lambda^\epsilon\in(0,\lambda_0)$, independent of $A_0$, such that for any $0<\lambda<\lambda^\epsilon$ there holds
\[
\langle \Po w_\lambda, w_\lambda\rangle\leqslant 4\pi^2(1+\epsilon)(A_0^2+1)\log(1/\lambda)+C_0,
\]
where $C_0$ does not depend on neither $\lambda$ nor $\epsilon$. Keep in mind that $s(\lambda) w_\lambda+c(\lambda) \in X_f^*$ with $s(\lambda)$ satisfying \eqref{bdofs}. Hence, we have
\begin{align*}
\beta_\lambda\leqslant & 2\langle \Po (s(\lambda)w_\lambda+c), s(\lambda)w_\lambda+c\rangle\\
=& 2s(\lambda)^2\langle \Po w_\lambda, w_\lambda\rangle\\
\leqslant & 32\pi^2(1+\epsilon)(A_0^2+1)\log(1/\lambda)+O(1).
\end{align*}
This implies that
$$\limsup_{\lambda\searrow0}\frac{\beta_\lambda}{\log(1/\lambda)}\leqslant32\pi^2(1+\epsilon)(A_0^2+1).$$
Letting $\epsilon \searrow 0$ and $A_0 \searrow 1$ gives the assertion.
\end{proof}

\begin{lemma}\label{lambdabetalambdaprime}
 There holds
 $$\liminf_{\lambda\searrow0}(\lambda\alpha_\lambda)\leqslant\liminf_{\lambda\searrow0}|\lambda\beta_\lambda^\prime|\leqslant64\pi^2.$$
\end{lemma}
\begin{proof}
Notice that the monotone function $\beta_\lambda$ is differentiable almost everywhere. Then by Lemma \ref{monotonicitybeta} we can easily get that
\[
\liminf_{\lambda\searrow0}(\lambda\alpha_\lambda)\leqslant\liminf_{\lambda\searrow0}|\lambda\beta_\lambda^\prime|.
\] 
So, it remains to show that
\[
\liminf_{\lambda\searrow0}|\lambda\beta_\lambda^\prime|\leqslant64\pi^2.
\] 
Indeed, if we assume that for some $0<\lambda_*<\lambda_0$, some $c_0>64\pi^2$ and almost all $0<\lambda<\lambda_*$ the absolutely continuous part of the differential of $\beta_\lambda$ satisfies
$|\beta^\prime_\lambda|\geqslant c_0/\lambda,$
then for $K=32\pi^2+c_0/2>64\pi^2$ and any sufficiently small $0<\lambda<\lambda_*$ we have, by Lebesgue's theorem, that
\begin{align*}
\beta_{\lambda}-\beta_{\lambda_*}\geqslant & \int_{\lambda}^{\lambda_*}|\beta_\lambda^\prime|~d\lambda \\
\geqslant & K\log(1/\lambda)+(c_0-K)\log(1/\lambda)+c_0\log\lambda_*\\
> & K\log(1/\lambda).
\end{align*}
 This contradicts the bound in Lemma \ref{bdofbeta}.
\end{proof}

With help of this lemma, we can now obtain a bound for the total $Q$-curvatures of the metric $\widetilde{g}_\lambda$ and the normalized metric $g_\lambda$. Recall that $\widetilde{g}_\lambda = e^{2\widetilde u_\lambda} g_0$ with $Q_{\widetilde{g}_\lambda}= f_\lambda$ and $g_\lambda = e^{2 u_\lambda} g_0$ with $Q_{g_\lambda}= \alpha_\lambda f_\lambda$.

\begin{lemma}\label{bdoftotalQcurvature}
 There holds
 $$\liminf_{\lambda\searrow0}\int_M |Q_{g_\lambda}| \dv_{g_\lambda}=\liminf_{\lambda\searrow0}\int_M |Q_{\widetilde{g}_\lambda}| \dv_{\widetilde{g}_\lambda}\leqslant128\pi^2.$$
\end{lemma}
\begin{proof}
 Notice that we can estimate
 \begin{equation}
 \label{bdoff_lambda}
 |Q_{\widetilde{g}_\lambda}|=|f_0 + \lambda |\leqslant-f_0+\lambda=-f_\lambda+2\lambda.
 \end{equation}
Keep in mind that $\alpha_\lambda = \vol(M, \widetilde g_\lambda)$ and that $\int_M f_\lambda \dv_{\widetilde{g}_\lambda} = 0$. Then by \eqref{alphalambda}, Lemma \ref{lambdabetalambdaprime}, and the fact that $u_\lambda\in X_{f_\lambda}^*$, we get that
\begin{align*}
\liminf_{\lambda\searrow0}\int_M|Q_{\widetilde{g}_\lambda}| \dv_{\widetilde{g}_\lambda}\leqslant &\liminf_{\lambda\searrow0}\Big[\int_M (-f_\lambda) \dv_{\widetilde{g}_\lambda}+2\lambda\alpha_\lambda\Big]\\
= & 2\liminf_{\lambda\searrow0}(\lambda\alpha_\lambda)\leqslant128\pi^2.
\end{align*}
Since $|Q_{g_\lambda}| =e^{4c_\lambda} |Q_{\widetilde g_\lambda}|$ and $\dv_{g_\lambda} = e^{4u_\lambda} \dvg = e^{-4c_\lambda}\dv_{\widetilde g_\lambda}$, we deduce that
\[
\int_M |Q_{g_\lambda}| \dv_{g_\lambda}=\int_M |Q_{\widetilde{g}_\lambda}| \dv_{\widetilde{g}_\lambda},
\] 
we thus complete the proof.
\end{proof}

\subsection{Concentration of curvature}

In the following, we consider the prescribed $Q$-curvature equation with an error term. To be precise, for a suitable sequence $\lambda_k\searrow0$ and suitable $\alpha_k > 0$ we let functions $w_k\in X_{f_{\lambda_k}}^*$ with corresponding metrics $g_k=e^{2w_k}g_0$ solve
\begin{equation} \label{pQeerror}
 \Po w_k= \alpha_kf_{\lambda_k}e^{4w_k}+h_ke^{4w_k}
\end{equation}
with $Q_{g_k} = \alpha_kf_{\lambda_k} + h_k$. Then
\[
\int_MQ_{g_k}e^{4w_k} \dvg =0.
\] 
In view of Lemma \ref{lambdabetalambdaprime}, we further assume that $\alpha_k$ satisfies 
\begin{equation} \label{Bound4LambdaAlpha}
\limsup_{k \to +\infty}(\lambda_k\alpha_k)\leqslant64\pi^2.
\end{equation}
Moreover, we let functions $h_k$ on $M$ be such that
\begin{equation} \label{pHkerror}
 \| h_k \| _{L^2(M,g_k)}=:\epsilon_k \to 0
\end{equation}
as $k \to +\infty$. Denote 
\begin{equation}
 \label{doublebarhk}
 \doverline{h}_k=\int_Mh_ke^{4w_k}~\dvg.
\end{equation}
Then the assumption \eqref{pHkerror} implies that
\begin{equation}
 \label{hkgozero}
 |\doverline{h}_k|\leqslant\epsilon_k\quad\mbox{and}\quad \|h_k-\doverline{h}_k\|_{L^1(M,g_k)}\leqslant2\epsilon_k.
\end{equation}

With all these assumptions, we then have the same conclusion as Lemma \ref{bdoftotalQcurvature}. To see this, we set
\[
s^{\pm}=\pm\max\{\pm s,0\}
\] 
for any $s\in\mathbb{R}$. Upon writing $|Q_{g_k}|=-Q_{g_k}+2Q^+_{g_k}$, estimating $Q^+_{g_k}\leqslant\alpha_k\lambda_k+|h_k|$, and integrating \eqref{pQeerror} we obtain, by H\"older's inequality and the assumption \eqref{pHkerror}, that
\begin{equation} \label{bdoftotalQcurvatureerror}
\begin{aligned}
\limsup_{k \to +\infty}\int_M|Q_{g_k}| \dv_{g_k}= &2\limsup_{k \to +\infty}\int_MQ^+_{g_k}e^{4w_k} \dvg \\
 \leqslant &2\limsup_{k \to +\infty}\big(\alpha_k\lambda_k+ \| h_k \| _{L^2(M,g_k)}\big)\leqslant128\pi^2.
\end{aligned}
\end{equation}
It is worth emphasizing that by allowing the ``error term'' $h_k$ in the perturbed equation \eqref{pQeerror}, we will also be able to apply Theorems \ref{main3} and \ref{main4} below in the flow context, where $w_k=u(t_k)$ for a solution $u=u(t)$ to \eqref{eeforu} and $h_t=u_t(t_k)$ for a sequence of times $t_k \to +\infty$. On the other hand, by choosing $w_k=u_k\in X^*_{f_{\lambda_k}}$, satisfying \eqref{pQeerror} with $h_k=0$ for all $k\in\mathbb{N}$, Theorems \ref{main1} and \ref{main} will become the special cases of Theorems \ref{main3} and \ref{main4} below respectively. 

It is worth noting that we are not interested in the existence of solutions to \eqref{pQeerror} in $X_{f_{\lambda_k}}^*$ under the conditions \eqref{Bound4LambdaAlpha} and \eqref{pHkerror}. What we are interested in is the concentration behavior of any sequence of solutions to \eqref{pQeerror} in $X_{f_{\lambda_k}}^*$, if exists. To be more precise, we prove the following concentration result.

\begin{lemma}\label{concentration0}
Given $(w_k)$ a sequence of solutions to \eqref{pQeerror} as above we have $\alpha_k\rightarrow+\infty$ as $k\rightarrow+\infty$. Moreover, there exist a suitable positive integer $I$ with $I\leqslant 8$ and finitely many points $x_\infty^{(i)} \in M$ with $1 \leqslant i \leqslant I$ such that, for any $r>0$ and each $1 \leqslant i \leqslant I$, there hold
\begin{equation}\label{ConcentrationPointIsMax}
f_0(x_\infty^{(i)})=0
\end{equation}
and
\begin{equation}\label{concentration}
 \liminf_{k \to +\infty}\int_{B_r(x_\infty^{(i)})}Q_{g_k}^+ \dv_{g_k}\geqslant8\pi^2.
\end{equation}
\end{lemma}

\begin{proof}
Our proof consists of two parts.

\medskip
\noindent\textbf{PART 1}. We prove \eqref{concentration} for $1 \leqslant i \leqslant I$ and $\alpha_k\rightarrow+\infty$ as $k\rightarrow+\infty$.

By way of contradiction, we assume that for every $x\in M$ there exists some $r_x>0$ such that 
 \begin{equation} \label{ubdofQ+}
 \int_{B_{r_x}(x)}Q_{g_k}^+ \dv_{g_k}\leqslant8\pi^2-\delta_x,
 \end{equation}
for some $\delta_x>0$ and for $k$ large enough. Since the proof presented here is rather long, we split it into several steps for clarity. 

\medskip
\noindent\textbf{Step 1}. In this step, from the contradiction assumption \eqref{ubdofQ+}, we shall establish the key estimate \eqref{bdofintvk+} below. Since $M$ is compact, we can cover $M$ by $N$ balls $B_i=B_{r_{x^i}/2}(x^i)$ with $1\leqslant i\leqslant N$. By the property (P2) of Green's function, we conclude that $\mathbb G(x,y)>0$ for any $x\in M$, $y\in B_{r_x}(x)$ with $r_x$ suitably small. So, in the following, we choose the radius $r_{x^i}$ small enough such that $\mathbb G(x,y)>0$ for any $x, y\in \widetilde{B}_i=B_{r_{x^i}}(x^i)$. Let
\[
\mu_k=\int_MQ^+_ke^{4w_k} \dvg .
\] 
Then the fact that $\int_MQ_{g_k}e^{4w_k} \dvg =0$ implies that $\int_MQ^-_ke^{4w_k} \dvg =-\mu_k$. Moreover, by \eqref{bdoftotalQcurvatureerror} we have
\[
0\leqslant\mu_k\leqslant128\pi^2+o(1).
\]
Now we let $w_k^{(\pm)}$ solve the equations
\begin{equation} \label{eqforuk+-}
 \Po w_k^{(\pm)}=Q_{g_k}^{\pm}e^{4w_k}\mp\mu_k
\end{equation}
on $M$. Since $\mbox{ker} \Po =\{\mbox{constants}\}$ and $ \Po w_k=Q_{g_k}e^{4w_k}$, we can decompose $w_k$ as the following
\begin{equation}
 \label{decomposeofuk}
 w_k=w_k^{(+)}+w_k^{(-)}+d_k,
\end{equation}
where $d_k$ is some constant. Integrating \eqref{decomposeofuk} with respect to the metric $g_0$ yields
\begin{equation}
 \label{dk}
 d_k=\overline w_k-\overline{w_k^{(+)}}-\overline{w_k^{(-)}}.
\end{equation}
Recall that $M$ is covered by all balls $B_i$ with $1\leqslant i\leqslant N$. Hence, for each $x\in M$ we can find some $1\leqslant i\leqslant N$ such that $x\in B_i$. By applying the formula \eqref{green} to the equations \eqref{eqforuk+-} we obtain
\begin{equation} \label{greenforuk+-}
\begin{aligned}
w_k^{(\pm)}(x)=&\overline{w_k^{(\pm)}}+\int_M\mathbb G(x,y) \big(Q_{g_k}^{\pm}e^{4w_k}\mp\mu_k\big)(y) \dvg (y) \\
=&\overline{w_k^{(\pm)}}+\int_{\widetilde{B}_i}\mathbb G(x,y)Q_{g_k}^{\pm} \dv_{g_k}(y)\mp\mu_k\int_M\mathbb G(x,y) \dvg (y) \\
&+\int_{M\backslash\widetilde{B}_i}\mathbb G(x,y)Q_{g_k}^{\pm} \dv_{g_k}(y).
\end{aligned}
\end{equation}
It follows from the property (P1) of Green's function and \eqref{bdoftotalQcurvatureerror} that there exists positive constants $C_i$ independent of $k$ such that
\begin{equation} \label{bdofGreenGreen}
\Big|\mu_k\int_M\mathbb G(x,y) \dvg (y)\Big|+\Big|\int_{M\backslash\widetilde{B}_i}\mathbb G(x,y)Q_{g_k}^{\pm} \dv_{g_k}(y)\Big|\leqslant C_i.
\end{equation}
If we set 
\[
c_*=\max \big\{ C_i : 1 \leqslant i \leqslant N \big \},
\] 
then the positivity of $G$ on each $\widetilde{B}_i$ and \eqref{greenforuk+-} imply that for any $x\in M$, we have
\begin{equation} \label{bdofuk+-}
\left\{
 \begin{aligned}
 w_k^{(+)}(x)&\geqslant \overline{w_k^{(+)}}-c_*,\\[0.5em]
 w_k^{(-)}(x)& \leqslant \overline{w_k^{(-)}}+c_*.
 \end{aligned}
\right.
\end{equation}
Now, we define
\begin{equation}\label{vkandvk+-}
\left\{
 \begin{aligned}
v_k&=w_k-\overline w_k, \\
v_k^{(+)}&=w_k^{(+)}- \overline{w_k^{(+)}}+c_*, \\
v_k^{(-)}&=w_k^{(-)}- \overline{w_k^{(-)}}-c_*.
 \end{aligned}
\right.
\end{equation}
Then the relations in \eqref{bdofuk+-} imply that
\begin{equation*}
v_k^{(+)}\geqslant 0 \geqslant v_k^{(-)}.
\end{equation*}
Furthermore, it follows from \eqref{decomposeofuk} and \eqref{dk} that
\[
w_k-\overline w_k=w_k^{(+)}+w_k^{(-)}-\overline{w_k^{(+)}}-\overline{w_k^{(-)}},
\] 
which then yields
\[
v_k=v_k^{(+)}+v_k^{(-)}.
\]
In view of \eqref{ubdofQ+}, we may choose some real number $s_0>4$ so that
\begin{equation}
 \label{ubdofQ+1}
 s_0\int_{\widetilde{B}_i}Q_{g_k}^+ \dv_{g_k}<32\pi^2
\end{equation}
for any $1\leqslant i\leqslant N$. By using \eqref{greenforuk+-} and \eqref{bdofGreenGreen}, we can bound
\begin{align*}
v_k^{(+)}-c_* =& w_k^{(+)}- \overline{w_k^{(+)}} \leqslant c_* + \int_{\widetilde{B}_i}\mathbb G(x,y)Q_{g_k}^{+} \dv_{g_k}(y),
\end{align*}
which then gives
\begin{align*}
\exp\big[s_0 v_k^{(+)} \big] & \leqslant e^{2s_0c_*}\exp\Big[s_0\int_{\widetilde B_i}\mathbb G(x,y)Q_{g_k}^{+} \dv_{g_k}(y)\Big] \\
& \leqslant e^{2s_0c_*}\exp\Big[s_0\int_M \| Q_{g_k}^+\chi_{\widetilde{B}_i} \| _{L^1(M,g_k)} |\mathbb G(x,y)| \frac{Q_{g_k}^{+} \chi_{\widetilde B_i} (y)}{\| Q_{g_k}^+\chi_{\widetilde{B}_i} \| _{L^1(M,g_k)}} \dv_{g_k}(y)\Big] \\
& \leqslant e^{2s_0c_*}\int_M \exp\Big[s_0 \| Q_{g_k}^+\chi_{\widetilde{B}_i} \| _{L^1(M,g_k)} |\mathbb G(x,y)| \Big] \frac{Q_{g_k}^{+} \chi_{\widetilde B_i} (y)}{\| Q_{g_k}^+\chi_{\widetilde{B}_i} \| _{L^1(M,g_k)}} \dv_{g_k}(y),
\end{align*}
thanks to the `weighted' Jensen inequality; see \cite[page 1227]{BM}. By integrating the inequality above and using Fubini's theorem, we obtain
\begin{align*}
\int_{B_i}e^{s_0v_k^{(+)}} \dvg (x)
&\leqslant e^{2s_0c_*}\int_M \Big( \int_{B_i}\exp\Big[s_0 \| Q_{g_k}^+\chi_{\widetilde{B}_i} \| _{L^1(M,g_k)}|\mathbb G(x,y)|\Big] \dvg (x) \Big) \\
& \times\frac{Q_{g_k}^+\chi_{\widetilde{B}_i}(y)}{ \| Q_{g_k}^+\chi_{\widetilde{B}_i} \| _{L^1(M,g_k)}} \dv_{g_k}(y) \\
&\leqslant e^{2s_0c_*}\sup_{y\in M}\int_M\exp\Big[s_0 \| Q_{g_k}^+\chi_{\widetilde{B}_i} \| _{L^1(M,g_k)}|\mathbb G(x,y)|\Big] \dvg (x).
\end{align*}
By the property (P2), we know that $|\mathbb G(x,y)| \leqslant (1/(8\pi^2))\log(1/d(x,y)) + \Cscr_{\mathbb G}$ for any $x \ne y$, which implies that
\[
s_0 \| Q_{g_k}^+\chi_{\widetilde{B}_i} \| _{L^1(M,g_k)}|\mathbb G(x,y)| 
\leqslant \frac {s_0}{8\pi^2} \| Q_{g_k}^+\chi_{\widetilde{B}_i} \| _{L^1(M,g_k)} \log \frac 1{d(x,y)}
+ 8\pi^2 \Cscr_{\mathbb G} s_0 
\]
for any $x \ne y$. From this we can estimate
\begin{align*}
\int_M\exp\Big[s_0 \| Q_{g_k}^+ & \chi_{\widetilde{B}_i} \| _{L^1(M,g_k)}|\mathbb G(x,y)|\Big] \dvg (x) \\
& \leqslant e^{8\pi^2 \Cscr_{\mathbb G} s_0} \int_M\big( d(x,y) \big)^{- (s_0 /8\pi^2) \| Q_{g_k}^+\chi_{\widetilde{B}_i} \| _{L^1(M,g_k)}} \dvg (x).
\end{align*}
The last integral in the preceding inequality is uniformly bounded because
\[
(s_0 /8\pi^2) \| Q_{g_k}^+\chi_{\widetilde{B}_i} \| _{L^1(M,g_k)} <4,
\]
thanks to \eqref{ubdofQ+1}. So we have shown that
\[
\int_{B_i}e^{s_0v_k^{(+)}} \dv_{g_k}<+\infty
\]
for $1\leqslant i\leqslant N$. Since $M$ is covered by finitely many $B_i$'s, we conclude that
\begin{equation}\label{bdofintvk+}
 \int_Me^{s_0v_k^{(+)}} \dv_{g_k}<+\infty.
\end{equation}
This completes the first step.

\medskip
\noindent\textbf{Step 2}. In this step, we claim from the key estimate \eqref{bdofintvk+} that $v_k^{(+)}$, defined in \eqref{vkandvk+-}, is uniformly bounded. To see this, we let $p=2s_0/(s_0+4)$. Then it follows from $s_0>4$ that 
\begin{equation}
 \label{propertyofp}
 1<p<2, \quad 4p<s_0.
\end{equation}
With this real number $p$, Minkowski's inequality, and \eqref{eqforuk+-}, we can estimate
\begin{align*}
\big\| \Po v_k^{(+)}\big\|_{L^p(M, g_0)}&= \big\| \Po w_k^{(+)}\big\|_{L^p(M, g_0)}\\
&=\big\|Q_{g_k}^+e^{4w_k}\mp\mu_k\big\|_{L^p(M, g_0)}\\
 &\leqslant \alpha_k\lambda_k\big\|e^{4w_k}\big\|_{L^p(M, g_0)}+\big\|h_ke^{4w_k}\big\|_{L^p(M, g_0)}+\mu_k\\
 & = I+II+128\pi^2+o(1)_{k \nearrow + \infty}.
\end{align*}

\smallskip
\noindent\underline{Estimate of $I$}: By Jensen's inequality and the fact $w_k \in X_{f_{\lambda_k}}^*$, we know that
\[
\overline w_k = \int_M w_k \dvg \leqslant \frac 14 \log \Big( \int_M e^{4w_k} \dvg \Big) = 0.
\] 
Hence, $w_k=v_k+\overline w_k\leqslant v_k\leqslant v_k^{(+)}$. This together with the fact that $\alpha_k\lambda_k\leqslant 64\pi^2+o(1)_{k \nearrow + \infty}$, \eqref{bdofintvk+} , \eqref{propertyofp}, and H\"older's inequality implies that
\begin{align*}
 I\leqslant65\pi^2\big\|e^{4v_k^{(+)}}\big\|_{L^p(M, g_0)}\leqslant65\pi^2\big\|e^{s_0v_k^{(+)}}\big\|_{L^1(M, g_0)}^{4/s_0}.
\end{align*}
Thus $I=O(1)_{k \nearrow + \infty}$.

\medskip
\noindent\underline{Estimate of $II$}: To estimate this term, we make use of H\"older's inequality and the facts that $1<p<2$ and $s_0=4p/(2-p)$ to get
\begin{align*}
 II& = \| h_ke^{4w_k} \| _{L^p(M,g_0)}=\bigg(\int_M|h_k|^pe^{2pw_k}e^{2pw_k} \dvg \bigg)^{1/p}\\
 &\leqslant\|h_k\|_{L^2(M,g_k)}\bigg(\int_Me^{s_0w_k} \dvg \bigg)^{1/p-1/2}.
\end{align*}
By \eqref{bdofintvk+} and \eqref{pHkerror}, we deduce that
\[
II = o(1)_{k \nearrow +\infty}.
\]
Combining the estimates of $I$ and $II$ gives 
\[
 \int_M| \Po v_k^{(+)}|^p \dvg <+\infty.
\]
In addition, it follows from \eqref{bdofintvk+} that
\[
v_k^{(+)}\in L^q(M,g_0)
\]
for any $q \geqslant 1$. Thus, by standard elliptic theory, we have shown that $v_k^{(+)}$ is bounded in $W^{4,p}(M,g_0)$ for some $p>1$. Again by Sobolev's embedding, we conclude that $v_k^{(+)}$ is bounded in $C^{0,\alpha}(M,g_0)$ for some $\alpha\in[0, 4-4/p]$. The claim is proved. 

\medskip
\noindent\textbf{Step 3}. In this step, we show that the sequence $(\alpha_k)$ is unbounded. Indeed, suppose that $(\alpha_k)$ is bounded, namely, $\alpha_k = O(1)_{k \to +\infty}$. Mimicking the argument used in \eqref{bdoftotalQcurvatureerror} to get
\[
0 \leqslant \int_MQ^+_{g_k}e^{4w_k} \dvg 
 \leqslant \alpha_k\lambda_k+ \| h_k \| _{L^2(M,g_k)} = o(1)_{ k \nearrow +\infty},
\]
which tells us that \eqref{ubdofQ+} holds at any point in $M$. From this, we repeat the arguments in Steps 1 and 2 to realize that $(v_k^{(+)})$ is uniformly bounded. It is now possible to bound $w_k$ uniformly from above as follows
\[
w_k=v_k+\overline w_k\leqslant v_k\leqslant v_k^{(+)}\leqslant C.
\] 
In view of the estimate $se^s\geqslant-1$ for $s\leqslant0$ we find that
\[
\alpha_kf_{\lambda_k}e^{4w_k}w_k\leqslant C\quad\mbox{and}\quad|e^{2w_k}w_k|\leqslant C.
\]
uniformly in $M$. But then by multiplying \eqref{pQeerror} with $w_k$ we obtain the bound
\begin{align*}
\beta_{\lambda_k}&\leqslant 2\langle \Po w_k, w_k\rangle \\
&\leqslant2\int_M\alpha_kf_{\lambda_k}e^{4w_k}w_k \dvg +2\int_Mh_ke^{4w_k}w_k \dvg \\
&\leqslant C+ 2\| h_k \| _{L^2(M,g_k)} \|e^{2w_k}w_k \| _{L^2(M,g_0)}\leqslant C,
\end{align*}
provided $k$ is large enough, which contradicts Lemma \ref{unbdofbeta}. Thus, $(\alpha_k)$ is unbounded as claimed.

\medskip
\noindent\textbf{Step 4}. Now, we are in position to obtain a contradiction and show that there are finitely many points $x_\infty^{(i)} \in M$ with $1 \leqslant i \leqslant I$ such that \eqref{concentration} holds. Keep in mind that $\alpha_k \nearrow +\infty$ as $k \to +\infty$. Depending on the size of $\overline w_k$, there are two possible cases as follows.

\noindent\underline{Case 1}. Suppose that $\overline w_k \to -\infty$ as $k \to +\infty$. By the uniform boundedness of $v_k^{(+)}$ established before, we have
\[
w_k=v_k+\overline w_k\leqslant v_k^{(+)}+\overline w_k\leqslant C+\overline w_k .
\]
Consequently, $w_k \to -\infty$ uniformly on $M$ as $k \to +\infty$. This contradicts the fact that $\int_Me^{4w_k} \dvg =1$.

\medskip
\noindent\underline{Case 2}. Suppose that $\overline w_k\geqslant-C$ for some positive constant $C$. In view of \eqref{bdoftotalQcurvatureerror}, we choose $\gamma=1/17$ so that
\[
\gamma\int_M|Q_{g_k}| \dv_{g_k}<8\pi^2
\]
holds for large $k$. This estimate plays a similar role as that of \eqref{ubdofQ+}. Therefore, we can repeat the previous argument to get that
\begin{equation} \label{bdofintvk-}
 \big\|e^{- s_1 \gamma v_k^{(-)}}\big\|_{L^1 (M,g_0)}<+\infty
\end{equation}
for some $s_1>4$. This together with \eqref{pQeerror} and $\vol (M, g_0)=1$ implies that
\[
\begin{aligned}
\int_M\big(\alpha_k & |f_{\lambda_k}| \big)^{2^{-1}} e^{2v_k-2\gamma v_k^{(-)}} \dvg \\
&\leqslant \Big(\int_M\alpha_k|f_{\lambda_k}|e^{4v_k} \dvg \Big)^{1/2} \Big(\int_Me^{-4\gamma v_k^{(-)}} \dvg \Big)^{1/2}\\
 &=e^{-2\overline w_k} \Big( \int_M\alpha_k|f_{\lambda_k}|e^{4w_k} \dvg \Big)^{1/2} \Big(\int_Me^{-4\gamma v_k^{(-)}} \dvg \Big)^{1/2}\\
 & \leqslant e^{-2\overline w_k} \Big( \int_M \big( |Q_{g_k}| + |h_k| \big) e^{4w_k} \dvg \Big)^{1/2} \Big(\int_Me^{-4\gamma v_k^{(-)}} \dvg \Big)^{1/2}\\
 & \leqslant e^{-2\overline w_k} \Big(\int_M|Q_{g_k}| \dv_{g_k} + \big\| h_k\big\| _{L^2(M,g_k)} \Big)^{1/2} \Big(\int_Me^{-4\gamma v_k^{(-)}} \dvg \Big)^{1/2}.
\end{aligned}
\]
This, the lower bound of $\overline w_k$, \eqref{pHkerror}, \eqref{bdoftotalQcurvatureerror}, and \eqref{bdofintvk-} imply that
\begin{equation}\label{alphakbd0}
\int_M\big(\alpha_k|f_{\lambda_k}|\big)^{2^{-1}} e^{2v_k-2\gamma v_k^{(-)}} \dvg \leqslant C
\end{equation}
for some constant $C>0$. For any integer $m\geqslant1$, thanks to \eqref{bdofintvk-} and \eqref{alphakbd0}, we do iteration to get that
\begin{equation}\label{alphakbd}
\begin{aligned}
 \int_M\big( & \alpha_k|f_{\lambda_k}|\big)^{2^{-m}}e^{2^{2-m}v_k-2\gamma v_k^{(-)}} \dvg \\
 & \leqslant\Big(\int_M\big(\alpha_k|f_{\lambda_k}|\big)^{2^{1-m}}e^{2^{3-m}v_k-2\gamma v_k^{(-)}} \dvg \Big)^{2^{-1}}\Big(\int_Me^{- 4\gamma v_k^{(-)}} \dvg \Big)^{2^{-1}} \\
 & \leqslant\cdots\leqslant C\Big(\int_M\big(\alpha_k|f_{\lambda_k}|\big)^{2^{-1}}e^{2v_k-2\gamma v_k^{(-)}} \dvg \Big)^{2^{1-m}}\leqslant C
\end{aligned}
\end{equation}
for some new constant $C>0$. By choosing $m\geqslant1$ large enough such that $2^{1-m}<\gamma$ and fixing it, we have
\[
2^{2-m}v_k-2\gamma v_k^{(-)}=2\big[2^{1-m}v_k^{(+)}+(2^{1-m}-\gamma)v_k^{(-)}\big]\geqslant2^{2-m}v_k^{(+)}\geqslant 0.
\]
This implies that
$$\liminf_{k \to +\infty}\int_M|f_{\lambda_k}|^{2^{-m}}e^{2^{2-m}v_k-2\gamma v_k^{(-)}} \dvg \geqslant\int_M|f_0|^{2^{-m}} \dvg >0.$$
Substituting this estimate into \eqref{alphakbd} gives 
$$\liminf_{k \to +\infty}\alpha_k^{2^{-m}}\leqslant C\Big(\int_M|f_0|^{2^{-m}} \dvg \Big)^{-1}<+\infty,$$
which contradicts the fact that $\alpha_k \to +\infty$ as $k \to +\infty$ established in Step 3. This contradiction implies that there exists at least one point $x\in M$ such that \eqref{concentration} holds. Moreover, the bound of total $Q$-curvature \eqref{bdoftotalQcurvatureerror} shows that there can only have finitely many points in $M$ such that \eqref{concentration} holds. Let us denote by $I$ the number of such points and for clarity these points will be denoted by $x_\infty^{(i)}$ with 1 $\leqslant i \leqslant I$. This completes PART 1.

\medskip
\noindent\textbf{PART 2}. Proof of $f_0(x_\infty^i)=0$ for $1\leqslant i\leqslant I$ and $I\leqslant8$.

Suppose that, for some $i \in \{1 ,2 ,..., I\}$, we have $f_0(x_\infty^{(i)})<0$. Then, on one hand, for sufficiently small $\epsilon>0$, we may find some $r>0$ such that $$f_{\lambda_k}\leqslant-\epsilon/2$$ on $B_r(x_\infty^i)$ for $k$ sufficiently large. On the other hand, again we make use of the estimate $Q^+_{g_k}\leqslant (\alpha_kf_{\lambda_k})^+ +|h_k| \leqslant |h_k|$ to get
\begin{align*}
 \int_{B_r(x_\infty^i)}Q_{g_k}^+ \dv_{g_k}& \leqslant\int_{B_r(x_\infty^i)} \big(\alpha_kf_{\lambda_k} \big)^+ \dv_{g_k}+\int_M|h_k| \dv_{g_k} 
\leqslant \| h_k \| _{L^2(M,g_k)}^2,
\end{align*}
thanks to H\"older's inequality and $\vol(M, g_k)=1$. Thus $\int_{B_r(x_\infty^i)}Q_{g_k}^+ \dv_{g_k} \to 0$ as $k \to +\infty$, which contradicts \eqref{concentration}. Thus, \eqref{ConcentrationPointIsMax} holds. Finally, the estimate $I \leqslant 8$ follows from the inequality
\begin{align*}
\limsup_{k \to +\infty}\int_MQ^+_{g_k}e^{4w_k} \dvg \leqslant 64 \pi^2
\end{align*}
in \eqref{bdoftotalQcurvatureerror} and the inequality \eqref{concentration}.
\end{proof}

An immediate consequence of Lemma \ref{concentration0} is the following 
\begin{equation}\label{positivelbandubofalphaklambdak}
 8\pi^2-o(1)_{k \nearrow +\infty} \leqslant\alpha_k\lambda_k\leqslant 64\pi^2+o(1)_{k \nearrow +\infty}.
\end{equation}


\subsection{Blow-up analysis}
\label{subsec-BlowUpAnalysis}

In this subsection, we derive the blow-up behavior for the functions $w_k$ in \eqref{pQeerror}, namely
\[
\Po w_k=Q_{g_k}e^{4w_k}= \big( \alpha_kf_{\lambda_k} +h_k \big) e^{4w_k}
\]
under the following two hypotheses
\[
\limsup_{k \to +\infty}(\lambda_k\alpha_k)\leqslant64\pi^2
\]
and
\[
 \| h_k\| _{L^2(M,g_k)} = o (1)_{k \nearrow +\infty}.
\] 
We also characterize the shape of the associated conformal metrics $g_k = e^{2w_k} g_0$ as $k \to +\infty$. 

First we consider the non-degenerate case.

\subsubsection{Non-degenerate case}

\begin{theorem}\label{main3} 
Assume that the Paneitz operator $ \Po$ is positive with kernel consisting of constant functions. Let $f_0\leqslant0$ be a smooth, non-constant function with $\max_Mf_0=0$ having only non-degenerate maximum points. Then for $w_k$ as in \eqref{pQeerror} above and suitable $I\in\mathbb{N}$, $r_k^{(i)}\searrow0$, $x_k^{(i)} \to x_\infty^{(i)}\in M$ with $f_0(x_\infty^{(i)})=0$, $1\leqslant i\leqslant I$, as $k \to +\infty$ the following hold:
\begin{enumerate}[label=\rm (\roman*)]
 \item $w_k \to -\infty$ locally uniformly on $M_\infty=M\backslash\{x_\infty^{(i)};1\leqslant i\leqslant I\}$. 
 \item In normal coordinates around $x_\infty^{(i)}$, set $z^{(i)}_k=\exp_{x_\infty^{(i)}}^{-1}(x_k^{(i)})$ and $\widetilde{w}_k=w_k\circ\exp_{x_\infty^{(i)}}$. Then for each $1\leqslant i\leqslant I$, either
\begin{itemize}
 \item[{\rm(a)}] $\limsup_{k \to +\infty}r_k^{(i)}/\sqrt{\lambda_k} = 0$ and
\begin{align*}
\widehat w_k(z):=
\widetilde{w}\big(z^{(i)}_k+r_k^{(i)}z\big)+&\log r_k^{(i)} \to \widehat w_\infty(z)
\end{align*}
strongly in $H_{\rm loc}^4(\mathbf R^4)$, where $\widehat w_\infty$, up to a translation and a scaling, is given by 
$$\widehat w_\infty(z)=\log\Big(\frac{4\sqrt{6}}{4\sqrt{6}+|z|^2}\Big)$$
and it induces a spherical metric
\[
\widehat{g}_\infty=e^{4\widehat w_\infty}g_{\mathbf R^4}
\] 
of $Q$-curvature
\[
Q_{\widehat{g}_\infty}\equiv1
\] 
on $\mathbf R^4$ and $1\leqslant I\leqslant 4$, or
 \item[{\rm (b)}] $\limsup_{k \to +\infty}r_k^{(i)}/\sqrt{\lambda_k}>0$ and
\begin{align*}
\widehat w_k(z):=
\widetilde{w}\big(z_k^{(i)}+r_k^{(i)}z\big)+\log(r_k^{(i)}) \to \widehat w_\infty(z)
\end{align*}
strongly in $H^4_{\rm loc}(\mathbf R^4)$, where $\widehat w_\infty$, up to a translation and a scaling, solves
$$\Delta^2_z\widehat w_\infty(z)=\Big(1+\frac12{\rm Hess}_{f_0}(x_\infty^{(i)})\big[z,z\big]\Big)e^{4\widehat w_\infty}.$$
In addition, the metric
\[
\widehat{g}_\infty=e^{4\widehat w_\infty}g_{\mathbf R^4}
\]
on $\mathbf R^4$ has finite volume and finite total $Q$-curvature
\[
Q_{\widehat{g}_\infty}(z)=1+\frac12{\rm Hess}_{f_0}(x_\infty^{(i)})\big[z,z\big]
\]
and $1 \leqslant I \leqslant 8$.
\end{itemize}
\end{enumerate}
\end{theorem}

\begin{proof}
Our proof consists of two parts.

\medskip
\noindent\textbf{PART 1}. We establish Part (i) of the theorem. Recall that
\[
M_\infty=M\backslash\{x_\infty^{(i)}: 1\leqslant i\leqslant I\}
\]
and let $x\in M_\infty$ be arbitrary. Then it follows from Lemma \ref{concentration0} that there exists a radius $r_x>0$ perhaps depending on $x$ such that for large $k$ we have
\[
\int_{B_{r_x}(x)}Q_{g_k}^+ \dv_{g_k}<8\pi^2.
\]
Following the same notations defined in \eqref{vkandvk+-}, we split 
\[
v_k=v_k^{(+)}+v_k^{(-)}.
\] 
Also, we let $\gamma=1/17$. Since the preceding estimate serves the same role as that of \eqref{ubdofQ+} in the proof of Lemma \ref{concentration0}, by repeating a similar argument used in the proof of Lemma \ref{concentration0} to get \eqref{bdofintvk+}, we find that
\[
\big\|e^{v_k^{(+)}}\big\|_{L^{s_x}(B_{r_x}(x))}+\big\|e^{-\gamma v_k^{(-)}}\big\|_{L^{s_x}(B_{r_x}(x))}\leqslant C
\]
for some $s_x>4$, which could also depend on $x$. 

Now, given any open subset $\Omega\subset\overline \Omega \subset M_\infty$, our aim in this part is to show that $w_k \to -\infty$ uniformly in $\Omega$. To this purpose, we first cover $\overline\Omega$ by finitely many balls $B_{r_j/2}(x_j)$, $1\leqslant j\leqslant N_0$, in such a way that for each ball $B_j :=B_{r_j}(x_j)$ with $1\leqslant j\leqslant N_0$ we still have
\[
\int_{B_j}Q_{g_k}^+ \dv_{g_k}<8\pi^2.
\]
Set $s = \min_{1 \leqslant j \leqslant N_0} s_{x_j}$. Then it follows from the uniform boundedness of $(v_k^{(+)})$ in $M$ that 
\[
\big\|e^{v_k^{(+)}}\big\|_{L^s(\Omega)}\leqslant \sum_{1\leqslant j\leqslant N_0}\big\|e^{v_k^{(+)}}\big\|_{L^s(B_j)}\leqslant C.
\]
We may assume that $\Omega$ is connected and large enough so that $\int_\Omega f_0 \dvg <0$. 
If there holds $\overline w_k \geqslant-C>-\infty$, then we may argue as in Case 2 of Step 4 in PART 1 of the proof of Lemma \ref{concentration0} to obtain
\[
\liminf_{k \to +\infty}\alpha_k^{2^{-m}}\leqslant C\Big(\int_\Omega|f_0|^{2^{-m}} \dvg \Big)^{-1}<+\infty,
\]
which contradicts the fact that $\alpha_k \to +\infty$ as $k \to +\infty$. Hence, we must have
\[
\overline w_k \to -\infty
\]
as $k \to +\infty$. Then it follows from the uniform boundedness of $(v_k^{(+)})$ that
\[
w_k=v_k+\overline w_k \leqslant v_k^{(+)}+\overline w_k \to -\infty
\]
as $k \to +\infty$. This finishes the proof of Part (i).
 
\medskip
 \noindent\textbf{PART 2}. Starting from now to the rest of the proof, we establish Part (ii) of the theorem, namely, the blow-up behavior near each point $x_\infty^{(i)}$ with $1 \leqslant i \leqslant I$. Since the proof of this part is rather long, we also divide it into several claims. Before doing so, we devote ourselves to preliminaries necessary for the blow-up analysis below. For simplicity, we denote $x_0=x_\infty^{(i)}$. Let $i_g$ be the injectivity radius of $M$. Clearly $i_g > 0$ since $M$ is compact and the restriction of $\exp_{x_0}$ to $\{X \in T_{x_0}M : \| X \|_{g_0}< i_g\}$ induces a diffeomorphism onto $B_{i_g}(x_0)$. Assimilating $(T_{x_0}M, g_0 (x_0))$ to $(\mathbf R^n, \dz^2)$ isometrically, one can then consider $\exp_{x_0}$ as a local chart around the point $x_0$. 
This allows us to select $\delta_0 \in (0, i_g/2)$ sufficiently small such that for all $x \in M$ and all $y, z \in \mathbf R^4$, if $|y|\leqslant \delta_0$ and $|z| \leqslant \delta_0$, then
\begin{equation}\label{RiemannianVsEuclidien}
\frac { |y-z|}2 \leqslant d_g (\exp_x (y), \exp_x (z) ) \leqslant 2 |y-z| ;
\end{equation}
see \cite[page 43]{DER}. Let $\delta < \min\{1,\delta_0\}$ and denote by $\widehat{B}_\delta(0)$ the open ball $\{ x \in \mathbf R^4 : |x| < \delta\}$ in $\mathbf R^4$. As always, we often use either a hat or a tilde to denote quantities in $\mathbf R^4$. We now consider the exponential map
\[
\exp_{x_0}:\widehat{B}_\delta(0) \to M
\]
with $\exp_{x_0}(0)=x_0$. We can also assume that $\delta > 0$ is chosen sufficiently small in order to guarantee that the only maxima of $f_0$ in $\exp_{x_0}(\widehat{B}_\delta(0))$ is $x_0$. Since $\exp_{x_0}$ is an isometric diffeomorphism onto $B_{i_g}(x_0)$, we deduce that
\begin{equation}\label{eqExponentialMap1}
\exp_{x_0} ( \widehat B_{ r}(0)) = B_{ r}(x_0))
\end{equation}
whenever $r < \delta_0 $, while by \eqref{RiemannianVsEuclidien} it is not hard to see that
\begin{equation}\label{eqExponentialMap2}
\exp_{x_0} ( \widehat B_{ r}( z)) \subset B_{2 r}(\exp_{x_0}( z))
\end{equation}
whenever $ |z| + r < \delta_0$ and that
\begin{equation}\label{eqExponentialMap3}
B_{ r}(\exp_{x_0}( z)) \subset \exp_{x_0} ( \widehat B_{ 2 r}( z))
\end{equation}
whenever $ |z| + r < \delta_0$. Combining \eqref{eqExponentialMap2} and \eqref{eqExponentialMap3} gives
\begin{equation}\label{eqExponentialMap}
\widehat B_{r/2} (\exp_{x_0}^{-1}(y)) 
\subset \exp_{x_0}^{-1} (B_r (y))
\subset \widehat B_{2r} (\exp_{x_0}^{-1}(y)),
\end{equation}
which is often used throughout the paper. Given a function $h$ on $M$ we denote
\[
\widetilde{h}=h\circ\exp_{x_0};
\] 
we also consider the pull-back metric
\[
\widetilde{g}_0=\exp_{x_0}^*g_0.
\] 
Since $x_0$ is a non-degenerate maxima of $f_0$, up to an action of the orthogonal group, we may assume that $\widetilde{f}_0$ has the following expansion
\begin{equation}\label{f0expan}
 \widetilde{f}_0(z)=-\sum_{i=1}^4a_i\big(z^i\big)^2+O(|z|^3),
\end{equation}
for any $z = (z^1,...,z^4) \in\widehat{B}_\delta(0)$ with $0<a_1\leqslant a_2\leqslant a_3\leqslant a_4$. If we choose $\delta$ even smaller, then we can further assume that
\[
-\frac {3a_4}2 |z|^2 \leqslant\widetilde{f}_0(z)\leqslant-\frac {a_1}2 |z|^2
\]
for all $z\in\widehat{B}_\delta(0)$. From now on let us consider large $k$ in such a way that
\[
 \lambda_k/a_1 < \delta^4 /24.
\] 
We also set 
\[
A_k=\big\{x\in M: \alpha_kf_{\lambda_k}(x)+\doverline h_k \geqslant0\big\} \cap \overline{B_\delta(x_0)},
\]
where $\doverline{h}_k$ is defined by \eqref{doublebarhk}. Clearly, 
\[
\big\{x\in M: \alpha_kf_{\lambda_k}(x)+\doverline h_k \geqslant0\big\} 
=
\Big\{x\in M: \lambda_k + \frac{\doverline h_k}{\alpha_k} \geqslant - f_0 (x)\Big\} .
\] 
Combining \eqref{hkgozero} and \eqref{positivelbandubofalphaklambdak} gives $\lambda_k + | \doverline h_k |/\alpha_k \leqslant (3/2)\lambda_k$. From this, the estimate $-\widetilde{f}_0(z)\geqslant (a_1/2) |z|^2$ in $\widehat{B}_\delta(0)$, and \eqref{eqExponentialMap1} we conclude that
\begin{equation}\label{bdofAk}
 \exp_{x_0}^{-1}(A_k)\subset \widehat{B}_{\sqrt{3\lambda_k/a_1}}(0).
\end{equation}

\medskip

\noindent{\bf Claim 1}. There exists a constant $\rho_0>0$ such that for each $\rho\in(0,\rho_0)$, there exists a sequence of positive numbers $(r_k)_k$ and a sequence of points $(x_k)_k\subset\overline{B_\delta(x_0)}$ satisfying 
\begin{subequations}\label{volumelowebd2}
\begin{empheq}[left=\empheqlbrace]{align}
\label{r_klambda_k}
&0<r_k \leqslant \sqrt{3\lambda_k/a_1},\\
\label{rho_k}
 &\int_{B_{r_k}(x_k)}e^{4w_k} \dvg=\rho,\\
\label{x_k->w_k->} 
&x_k\rightarrow x_0\quad\mbox{and}\quad w_k(x_k)\rightarrow+\infty, ~~\mbox{as}~k\rightarrow\infty,\\
\label{integralw_k<=rho}
&\int_{B_{r_k}(y)}e^{4w_k} \dvg\leqslant\rho, \quad \text{ for all }~{y\in B_{\sqrt{r_k}}(x_k)}.
\end{empheq}
\end{subequations}

\noindent{\itshape Proof of Claim 1}. 
It follows from Lemma \ref{concentration0} and $Q^+_{g_k}\leqslant (\alpha_k f_{\lambda_k}+\doverline h_k )^+ + |h_k - \doverline h_k|$ that
\begin{align*}
 8\pi^2-o(1)_{k \nearrow +\infty}& \leqslant 
 \int_{B_\delta(x_0)} (\alpha_k f_{\lambda_k}+\doverline h_k )^+ e^{4 w_k} \dvg + \int_{B_\delta(x_0)} |h_k - \doverline h_k | e^{4 w_k} \dvg .
\end{align*}
Making use of \eqref{hkgozero} we further obtain
\[
 8\pi^2-o(1)_{k \nearrow +\infty} \leqslant \int_{B_\delta(x_0)} (\alpha_k f_{\lambda_k}+\doverline h_k )^+ e^{4 w_k} \dvg + o(1)_{k \nearrow +\infty} .
\]
On the other hand, by \eqref{bdofAk}, \eqref{hkgozero}, and \eqref{positivelbandubofalphaklambdak} we can estimate
\begin{align*}
 \int_{B_\delta(x_0)} (\alpha_k f_{\lambda_k}+\doverline h_k )^+ e^{4 w_k} \dvg &= \int_{A_k} (\alpha_k f_{\lambda_k}+\doverline h_k ) e^{4 w_k} \dvg \\
&= \int_{\exp_{x_0}^{-1}(A_k)} (\alpha_k \widetilde f_{\lambda_k}+\doverline h_k ) e^{4\widetilde{w}_k} \dv_{\widetilde{g}_0} \\
&\leqslant \int_{\exp_{x_0}^{-1}(A_k)} \big(\alpha_k\lambda_k+o(1)_{k \nearrow +\infty}\big)e^{4\widetilde{w}_k} \dv_{\widetilde{g}_0} \\
&\leqslant\big(64\pi^2+o(1)_{k \nearrow +\infty}\big)\int_{B_{\sqrt{3\lambda_k/a_1}}(x_0)}e^{4w_k} \dvg.
\end{align*}
Putting all these estimates together, we eventually get
\begin{equation*}
\begin{aligned}
 8\pi^2-o(1)_{k \nearrow +\infty}& \leqslant \big(64\pi^2+o(1)_{k \nearrow +\infty}\big)\int_{B_{\sqrt{3\lambda_k/a_1}}(x_0)}e^{4w_k } \dvg.
\end{aligned}
\end{equation*}
Hence, we have just shown that
\begin{equation}\label{volumelowerbd1}
 \int_{B_{\sqrt{3\lambda_k/a_1}}(x_0)}e^{4w_k} \dvg \geqslant\frac{7}{65}.
\end{equation}
Now, if we let
\[
\tau(s)=\sup_{x\in \overline{B_\delta(x_0)}}\int_{B_s(x)}e^{4w_k} \dvg,
\]
then $\tau(0)=0$ and by \eqref{volumelowerbd1} we know that
\[
\tau(\sqrt{3\lambda_k/a_1})\geqslant \frac 7{65}.
\] 
By setting $\rho_0=7/65$ and by the continuity of $\tau$, we thus have for each $\rho\in(0,\rho_0)$, there exists some $r_k\in(0, \sqrt{3\lambda_k/a_1})$ such that $\tau(r_k)=\rho$. Furthermore, the compactness of $\overline{B_{\delta}(x_0)}$ allows us to choose $x_k\in \overline{B_{\delta}(x_0)}$ such that 
\[
\int_{B_{r_k}(x_k)}e^{4 w_k} \dvg =\sup_{x\in \overline{B_\delta(x_0)}}\int_{B_{r_k}(x)}e^{4w_k} \dvg
\]
for each $k\in \mathbb{N}$. This finishes the proof of \eqref{r_klambda_k} and \eqref{rho_k}. 

Next, let us show \eqref{x_k->w_k->}. To see this, we assume by contradiction that $w_k(x_k)\leqslant C_w$ for some constant $C_w>0$. On one hand, by the estimate $(\alpha_kf_{\lambda_k}+\doverline h_k )^+ \geqslant (\alpha_kf_{\lambda_k}+ h_k)^+ - |h_k - \doverline h_k|$, Lemma \ref{concentration0}, and \eqref{hkgozero}, we get
\begin{align*}
\liminf_{k \to +\infty} & \int_{B_\delta(x_0)}\big(\alpha_kf_{\lambda_k}+\doverline h_k\big)^+e^{4w_k} \dvg \\
 & \geqslant\liminf_{k \to +\infty}\Big(\int_{B_\delta(x_0)}(\alpha_kf_{\lambda_k}+h_k)^+e^{4w_k} \dvg - \| h_k-\doverline h_k \| _{L^1(M,g_k)}\Big)\\
 & \geqslant8\pi^2.
\end{align*}
On the other hand, we have, by \eqref{positivelbandubofalphaklambdak} and \eqref{hkgozero}, that
\begin{align*}
\liminf_{k \to +\infty}\int_{B_\delta(x_0)} (\alpha_kf_{\lambda_k}+\doverline h_k)^+e^{4w_k} \dvg
&=\liminf_{k \to +\infty}\int_{A_k}(\alpha_kf_{\lambda_k}+\doverline h_k)e^{4w_k} \dvg\\
& \leqslant\liminf_{k \to +\infty}\Big(e^{4C_w} \alpha_k\lambda_k\int_{A_k} \dvg \Big)\\
& \leqslant O(\delta^4).
\end{align*}
We thus obtain a contradiction if we choose $\delta$ small at the beginning. Thus, we have already established the unboundedness of $w_k (x_k)$. To see why $x_k \to x_0$ as $k \to +\infty$, we assume by contradiction that $x_k \to x_*\neq x_0$. Clearly, $x_*\in M_\infty$ because $(x_k) \subset \overline{B_\delta (x_0)}$. By the result in Part (i) we know that $w_k(x_*) \to -\infty$ which contradicts the fact that $w_k(x_k) \to +\infty$.

Finally, keep in mind that $r_k < \sqrt{3\lambda_k/a_1}< (\delta/2)^2$. This together with the proved fact \eqref{x_k->w_k->} above immediately implies that $B_{\sqrt{r_k}}(x_k) \subset B_{\delta}(x_0)$. Hence from our choice of $\rho$ we have
\[
\int_{B_{r_k}(y)}e^{4 w_k} \dvg \leqslant\rho
\]
for all $y\in B_{\sqrt{r_k}}(x_k)$. Thus \eqref{integralw_k<=rho} is proved and we complete the proof of Claim 1. \qed

\medskip

Set $z_k=\exp_{x_0}^{-1}(x_k).$ With the choice of $r_k$ and $z_k$ above, we consider in $\mathbf R^4$ the translation--dilation
\begin{align*}
\Gamma_k: z & \mapsto z_k+r_kz.
\end{align*}
Clearly, $\Gamma_k (\widehat D_{k,\delta})= \widehat B_{\delta}(0)$, where, given $r>0$, the set $\widehat D_{k,r}$ is defined as follows
\[
\widehat D_{k,r} := \{ z \in \mathbf R^4 : |z_k + r_k z| < r \}. 
\]
Clearly we may rewrite $\widehat D_{k,r}$ as $\widehat D_{k,r}=\widehat B_{r/r_k}(-z_k/r_k)$. Recall that $r_k \to 0$ and $z_k \to 0$ as $k \to +\infty$ by Claim 1. This implies that $z_k/r_k = o(r/r_k)_{k \nearrow +\infty}$. From this we deduce that, for each $r>0$ fixed, the set $\widehat D_{k,r}$ exhausts $\mathbf R^4$ as $k \to +\infty$. Next, we consider the scaled metrics 
\[
\widehat g_k=r_k^{-2}\Gamma_k^*\widetilde{g}_0
\]
on $\widehat D_{k,r}$. Also, we define 
\begin{equation}\label{eqWHat_k}
\widehat w_k=\widetilde{w}_k \circ \Gamma_k +\log r_k.
\end{equation}
Making use of \eqref{eqExponentialMap} gives
\begin{equation}\label{eqExponentialMap4}
\widehat B_{R/2} (0) 
\subset (\exp_{x_0} \circ \Gamma_k) ^{-1} (B_{Rr_k} (x_k))
\subset \widehat B_{2R} (0).
\end{equation}
 In view of the conformally covariant property of $\PP$, there holds $\PP_{\widehat g_k} = r_k^4 \PP_{\Gamma_k^*\widetilde g_0}$. Then by a direct computation, the function $\widehat w_k$ satisfies
 \begin{equation}\label{eqBeforeLimitingEquation}
 \begin{aligned}
 \PP_{\widehat g_k}\widehat w_k(z) & = r_k^4 \PP_{\Gamma_k^*\widetilde g_0} \big( \widetilde w_k(\Gamma_k(z)) \big)\\
 & = \big(\alpha_k\widetilde{f}_{\lambda_k}(\Gamma_k(z))+\widetilde{h}_k\circ\Gamma_k\big)e^{4[ \widetilde w_k(\Gamma_k(z)) + \log r_k ]} \\
& =\widehat f_k(z) e^{4\widehat{w}_k(z)},
 \end{aligned}
 \end{equation}
 where 
 \begin{equation*}
 \widehat f_k =\alpha_k\widetilde{f}_{\lambda_k} \circ \Gamma_k +\widetilde{h}_k\circ\Gamma_k.
 \end{equation*}
Using the exponential map, we can rewrite the identity in \eqref{rho_k} as follows
\begin{equation}\label{volumeconcentration2}
 \rho=\int_{(\exp_{x_0}\circ\Gamma_k)^{-1}(B_{r_k}(x_k))} e^{4\widehat w_k} \dv_{\widehat g_k} .
\end{equation}
To rewrite the inequality in \eqref{integralw_k<=rho}, first we make use of \eqref{eqExponentialMap} to get
\[
\widehat B_{r_k/2} (\exp_{x_0}^{-1}(y)) \subset \exp_{x_0}^{-1} (B_{r_k} (y)).
\]
Since $\Gamma_k^{-1} : z \mapsto (z-z_k)/r_k$, we deduce that
\[
\widehat B_{1/2} \Big( \frac{\exp_{x_0}^{-1}(y) - z_k}{r_k} \Big) \subset (\exp_{x_0}\circ\Gamma_k)^{-1} (B_{r_k} (y)).
\]
Therefore, the last inequality in \eqref{volumelowebd2} gives
$$\int_{\widehat B_{1/2} \big( \frac{\exp_{x_0}^{-1}(y) - z_k}{r_k} \big)} e^{4\widehat w_k} \dv_{\widehat g_k} \leqslant \rho$$
for all $y\in B_{\sqrt{r_k}}(x_k)$. By \eqref{eqExponentialMap4}, notice that
\[
(\exp_{x_0}\circ\Gamma_k)\big(\widehat B_{1/(2\sqrt{r_k})}(0)\big)=\exp_{x_0}\big(\widehat B_{\sqrt{r_k}/2}(z_k)\big)\subset B_{\sqrt{r_k}}(x_k).
\]
Hence, substituting $y=\exp_{x_0}(\Gamma_k(z))$ into the inequality above yields
\begin{equation}\label{volumeconcentration3}
\int_{\widehat B_{1/2}(z)} e^{4\widehat w_k} \dv_{\widehat g_k} \leqslant \rho
\end{equation}
for any $z \in \widehat B_{1/(2\sqrt{r_k})}(0)$. Since the set $\widehat{B}_{1/(2\sqrt{r_k})}(0)$ exhausts $\mathbf R^4$ as $k \to +\infty$, we can freely use \eqref{volumeconcentration3} for arbitrary $z$ in any fixed ball provided $k$ is suitably large. In the next step, we provide a more precise estimate on $ d(x_k,x_0)$ in terms of $\lambda_k$.

\medskip

\noindent{\bf Claim 2}. There exists some constant $C>0$ such that 
\begin{equation} \label{bdofdistxkandx0}
 d(x_k,x_0)\leqslant C\sqrt{\lambda_k}
\end{equation}
for all $k$ large. 

\noindent{\itshape Proof of Claim 2}. If this were not true, then we would have $d(x_k,x_0)/\sqrt{\lambda_k} \to +\infty$ as $k \to +\infty$. From the expansion of $\widetilde f_0$ in \eqref{f0expan}, the bound for $\alpha_k\lambda_k$ in \eqref{positivelbandubofalphaklambdak}, and the inequality $r_k^2/\lambda_k \leqslant 3/a_1$ we obtain for any fixed $R>2$ with $|z|\leqslant R$
\begin{align*}
 \alpha_k|\widetilde{f}_0(\Gamma_k(z))| 
 &\geqslant\big(\alpha_k\lambda_k\big)\lambda_k^{-1}\Big(\frac{a_1}{2}\big|z_k+r_kz\big|^2\Big)\\
 &\geqslant\big(\alpha_k\lambda_k\big)\lambda_k^{-1}\frac{a_1}{2}\Big( \frac12|z_k|^2-r_k^2|z|^2 \Big)\\
 &\geqslant 7\pi^2\Big(\frac{a_1}{4}\frac{|z_k|^2}{\lambda_k}-\frac32 R^2\Big) .
\end{align*}
This together with the fact that $|z_k|/\sqrt{\lambda_k} \to +\infty$ as $k \to +\infty$ implies that
\begin{equation*}
 \alpha_k|\widetilde{f}_0(\Gamma_k(z))| \to +\infty
\end{equation*}
uniformly in the ball $\widehat{B}_{R}(0)$. Thus for $K=65\pi^2/\rho$ there holds $\alpha_k|\widetilde{f}_0(\Gamma_k(z))|\geqslant K$ for all $z\in \widehat{B}_{R}(0)$ provided $k$ is large enough. Note that $|f_0| = \lambda_k - f_{\lambda_k}$. From this we may write
$$\alpha_k|\widetilde{f}_0(\Gamma_k(z))|=\alpha_k\lambda_k-\alpha_k\widetilde{f}_{\lambda_k}(\Gamma_k(z)).$$
Then by \eqref{rho_k}, $R>2$, \eqref{eqExponentialMap4}, and \eqref{hkgozero}, for $k$ large enough, we have the estimate
\begin{align*}
 65\pi^2&\leqslant K\int_{B_{Rr_k/2}(x_k)}e^{4w_k} \dvg \\
&=K\int_{(\exp_{x_0}\circ\Gamma_k)^{-1}(B_{Rr_k/2}(x_k))} e^{4\widehat w_k} \dv_{\widehat g_k}\\
&\leqslant\int_{\widehat{B}_R(0)}\alpha_k|\widetilde{f}_0(\Gamma_k(z))|e^{4\widehat w_k} \dv_{\widehat{g}_k}\\
 &= \int_{\widehat{B}_R(0)}\big(\alpha_k\lambda_k-\alpha_k\widetilde{f}_{\lambda_k}(\Gamma_k(z))\big)e^{4\widehat w_k} \dv_{\widehat{g}_k}\\
 &\leqslant \int_{(\exp_{x_0}\circ\Gamma_k)^{-1}(B_{2Rr_k}(x_k))} \big(\alpha_k\lambda_k-\alpha_k\widetilde{f}_{\lambda_k}(\Gamma_k(z))\big) e^{4\widehat w_k} \dv_{\widehat g_k}\\
 & = \int_{B_{2Rr_k}(x_k)} \big(\alpha_k\lambda_k- \alpha_kf_{\lambda_k} \big)e^{4w_k} \dvg \\
 &\leqslant \int_M \big(\alpha_k\lambda_k-(\alpha_kf_{\lambda_k}+h_k)\big)e^{4w_k} \dvg +o(1)\\
 &\leqslant \alpha_k\lambda_k+o(1) \\
 &\leqslant64\pi^2+o(1),
\end{align*}
which is impossible for $k$ sufficiently large. This proves \eqref{bdofdistxkandx0}. \qed

\smallskip

Using the expansion of $\widetilde f_0$ in \eqref{f0expan} we may write $\alpha_k\widetilde f_{\lambda_k}\circ\Gamma_k$ as
\begin{equation} \label{widehatfk}
\begin{aligned}
 \alpha_k\widetilde f_{\lambda_k}\big(\Gamma_k(z)\big)&=\alpha_k\lambda_k+\alpha_k\widetilde{f}_0(\Gamma_k(z)) \\
 &= \alpha_k\lambda_k\big[\lambda_k^{-1}\widetilde{f}_0(\Gamma_k(z))+1\big]\\
 &= \alpha_k\lambda_k
\left[
\begin{aligned}
&-\sum_{i=1}^4a_i\Big(\frac{z_k^i}{\sqrt{\lambda_k}}+\frac{r_k}{\sqrt{\lambda_k}}z^i\Big)^2\\
&+O\Big(\sqrt{\lambda_k}\Big|\frac{z_k}{\sqrt{\lambda_k}}+\frac{r_k}{\sqrt{\lambda_k}}z\Big|^3\Big) +1
\end{aligned}
\right].
\end{aligned}
\end{equation}
It follows from \eqref{volumelowebd2} and \eqref{bdofdistxkandx0} that 
$$\frac{r_k}{\sqrt{\lambda_k}}\leqslant \sqrt{3/a_1}\quad\mbox{and}\quad \frac{|z_k|}{\sqrt{\lambda_k}}\leqslant C.$$
Plugging these into \eqref{widehatfk} and using \eqref{positivelbandubofalphaklambdak} we can find a positive constant $C_R$ such that
\begin{equation}\label{uniformbdofwidehatfk}
 \alpha_k\big|\widetilde f_{\lambda_k}(\Gamma_k(z))\big|\leqslant C_R
\end{equation}
for any $z\in \widehat B_R(0)$.

\medskip
\noindent{\bf Claim 3}. Let $\widehat w_k$ be given in \eqref{eqWHat_k}. Then $\widehat{w}_k$ is bounded in $W^{4,s_0}_{\rm loc}(\mathbf R^4)$ for some $s_0>1$. Thus, there exists a function $\widehat w_\infty$ such that $\widehat{w}_k \to \widehat w_\infty$ strongly in $C^{0,\alpha}_{\rm loc}(\mathbf R^4)\cap H^2_{\rm loc}(\mathbf R^4)$ for any $0<\alpha< 1-1/s_0$. Moreover, there holds
\begin{equation}
 \label{volumeofhatWinftybd}
 \int_{\mathbf R^4}e^{4\widehat w_\infty} \dz\leqslant1.
\end{equation}

\noindent{\itshape Proof of Claim 3}: 
We borrow the method used in the proof of \cite[Proposition 3.4]{Ma}. Let $R>8$ be arbitrary but fixed. Then we define a smooth cut-off function $\eta_R$ with
\[
\eta_R (z)=
\begin{cases}
 1 &\mbox{if}~z \in \widehat B_{R/2}(0),\\
 0 &\mbox{if}~ z \in \mathbf R^4 \setminus \widehat B_{2R}(0).
\end{cases}
\]
Set
\[
\left\{
\begin{aligned}
a_k&=\frac{1}{|\widehat B_{R}(0)|}\int_{\widehat B_{R}(0)}\widehat w_k \dv_{\widehat g_k}, &\\
\Phi_k&=\eta_R\widehat w_k+(1-\eta_R)a_k, &\mbox{ and}\\
\widehat{\Phi}_k&=\Phi_k-a_k. &
\end{aligned}
\right.
\]
Then
\[
\Phi_k=
\begin{cases}
 \widehat w_k &\mbox{on}~\widehat B_{R/2}(0),\\
 a_k &\mbox{on}~\mathbf R^4 \setminus \widehat B_{2R}(0).
\end{cases}
\]
In particular, $\widehat{\Phi}_k=0$ in $\mathbf R^4\setminus \widehat B_{2R}(0)$. Hence, $\widehat{\Phi}_k$ has a uniform compact support. Observe that $\widehat \Phi_k =\eta_R \big( \widehat w_k - a_k)$. From this and the equation satisfied by $ \widehat w_k$ in \eqref{eqBeforeLimitingEquation}, it is not hard to see that $\widehat{\Phi}_k$ satisfies the following equation
\begin{equation}\label{eqforhatwk}
\PP_{\widehat g_k}\widehat{\Phi}_k=\eta_R\PP_{\widehat g_k} \widehat w_k +L_k\big(\widehat w_k-a_k\big) = \varphi_k,
\end{equation}
where
\[
\varphi_k=\eta_R\widehat f_ke^{4\widehat{w}_k}+L_k\big(\widehat w_k-a_k\big).
\]
Note that in \eqref{eqforhatwk}, $(L_k)_k$ are linear operators containing derivatives of order $0, 1, 2$ and $3$ with uniformly bounded and smooth coefficients. Therefore, by Lemma \ref{Mal06Lemma2.3} and some scaling argument one can easily find that
\begin{equation}\label{hatwkw3sbound}
\int_{\widehat B_{2R} (0)}\big(|\nabla^3\widehat{w}_k|^s+|\nabla^2\widehat{w}_k|^s+|\nabla\widehat{w}_k|^s\big) \dv_{\widehat g_k}\leqslant C_R
\end{equation}
for any $k\in\mathbb{N}$ and any $s\in[1,4/3)$. Since $\widehat{\Phi}_k$ has compact support and $\widehat{\Phi}_k = \widehat w_k - a_k$ in $\widehat B_{R/2} (0)$, we can apply $L^s$-Poincar\'e's inequality to get
\begin{equation}\label{hatwklsbound}
 \int_{\widehat B_{R/2}(0)}|\widehat{\Phi}_k|^s \dv_{\widehat g_k}\leqslant C_R
\end{equation}
for any $k\in\mathbb{N}$ and any $s\in[1,4/3)$. It follows from \eqref{hatwkw3sbound} and \eqref{hatwklsbound} that
\begin{equation}\label{Lklsbound}
\int_{\widehat B_{R/2} (0)}\big|L_k\big(\widehat w_k-a_k\big)\big|^s \dv_{\widehat g_k}\leqslant C_R
\end{equation}
for any $k\in\mathbb{N}$ and any $s\in[1,4/3)$. We now use \eqref{Lklsbound} together with H\"older's inequality to conclude, for any $z\in \widehat B_{R/4}(0)$ and $r>0$ sufficiently small, that
\begin{equation}\label{integralboundforLk}
\int_{\widehat{B}_{r}(z)}\big|L_k\big(\widehat w_k-a_k\big)\big| \dv_{\widehat g_k}\leqslant O(r).
\end{equation}
On the other hand, it follows from the boundedness of $\|h_k\|_{L^2(M,g_k)}$ in \eqref{pHkerror}, the boundedness of $\alpha_k \widetilde f_{\lambda_k} \circ \Gamma_k$ in \eqref{uniformbdofwidehatfk}, and the estimate of $\int e^{4\widehat{w}_k} \dv_{\widehat g_k}$ in \eqref{volumeconcentration3} that
\begin{equation}\label{integralboundforhatQ+}
\begin{aligned}
\int_{\widehat{B}_r(z)}\big| \widehat{f}_ke^{4\widehat{w}_k}\big| \dv_{\widehat g_k} 
 & \leqslant \int_{\widehat{B}_r(z)}\alpha_k\big|\widetilde{f}_{\lambda_k}\circ\Gamma_k\big|e^{4\widehat{w}_k} \dv_{\widehat g_k}+\int_{\widehat{B}_r(z)}\big|\widetilde{h}_k\circ\Gamma_k\big|e^{4\widehat{w}_k}~\dv_{\widehat g_k}\\
 & \leqslant C_R\int_{\widehat{B}_r(z)}e^{4\widehat{w}_k} \dv_{\widehat g_k} +\big\|h_k\big\|_{L^2(M,g_k)}\\
 & \leqslant C_R\rho+\epsilon_k\leqslant 2C_R\rho\\
\end{aligned}
\end{equation}
for any $z\in \widehat B_{R/4}(0)$, $r>0$ small, and $k$ large. Hence, by choosing $r>0$ and $\rho>0$ suitably small, we obtain from \eqref{integralboundforLk} and \eqref{integralboundforhatQ+} the following estimate
\[
\int_{\widehat{B}_r(z)}|\varphi_k| \dv_{\widehat g_k}<8\pi^2
\]
for all $z\in \widehat B_{R/4}(0)$. Then, it follows from the equation solved by $\widehat \Phi_k$ in \eqref{eqforhatwk}, Remark \ref{Prop3.1} and a finite covering argument that there exists some $s_1>1$ such that
\begin{equation}\label{e4hatwklsbound}
\int_{\widehat B_{R/4}(0)}e^{4s_1\widehat{\Phi}_k} \dv_{\widehat g_k}\leqslant C,
\end{equation}
where $C>0$ is a fixed constant.

Next, we show that $a_k$ is bounded. To see this, it follows from Jensen's inequality, \eqref{eqExponentialMap2} and the fact that $\int_M e^{ 4w_k } \dvg=1$ that
\begin{align*}
 a_k&= \frac{1}{|\widehat B_{R}(0)|}\int_{ \widehat B_{R}(0)}\widehat w_k \dv_{\widehat g_k} \\
&\leqslant\frac14\log\Big(\frac{1}{|\widehat B_{R}(0) |}\int_{\widehat B_{R}(0)}e^{4\widehat w_k} \dv_{\widehat g_k}\Big)\\
&\leqslant\frac14\log\Big(\frac{1}{|\widehat B_{R}(0) |}\int_{B_{2Rr_k}(x_k)}e^{4w_k} \dvg\Big) \leqslant C_R.
\end{align*}
To bound $a_k$ from below, we recall from \eqref{volumeconcentration2} the following
\[
\int_{(\exp_{x_0}\circ\Gamma_k)^{-1}(B_{r_k}(x_k))} e^{4\widehat w_k} \dv_{\widehat g_k} 
= \rho.
\]
Making use of \eqref{eqExponentialMap4} gives
\[
 (\exp_{x_0}\circ\Gamma_k)^{-1}(B_{r_k}(x_k)) \subset \widehat B_2 (0) .
\]
Consequently, for $k$ large and because $R/4>2$, we arrive at
\[
\int_{\widehat B_{R/4}(0)} e^{4\widehat w_k} \dv_{\widehat g_k} \geqslant \rho.
\] 
This together with the fact that $\Phi_k=\widehat w_k$ in $\widehat B_{R/4}(0)$, we obtain
\begin{align*}
\rho & \leqslant \int_{ \widehat B_{R/4} (0)}e^{4\Phi_k} \dv_{\widehat g_k}
 = e^{4 a_k}\int_{ \widehat B_{R/4} (0) }e^{4\widehat \Phi_k} \dv_{\widehat g_k},
\end{align*}
which implies by \eqref{e4hatwklsbound} that $a_k\geqslant -C_R$ and hence we find
\[
|a_k|\leqslant C_R.
\]
Using this fact, we have by \eqref{hatwklsbound} and Minkowski's inequality that
\begin{equation}\label{hatuk+lsbound}
 \int_{ \widehat B_{R/2}(0) }\big|\widehat w_k\big|^s \dv_{\widehat g_k}\leqslant C_R
\end{equation}
for all $s\in [1,4/3)$ and by \eqref{e4hatwklsbound} that
\begin{equation}\label{e4hatuk+lsbound}
 \int_{ \widehat B_{R/4} (0)}e^{4s_1\widehat w_k} \dv_{\widehat g_k}\leqslant C_R.
\end{equation}
Now, we take $1<s_2< \min\{s_1,2\}$ and let $s_0=2s_2/(1+s_2)$. Then $1<s_0<\min\{4/3,s_1\}$ and $s_2=s_0/(2-s_0)$. Using the boundedness of $\alpha_k \widetilde f_{\lambda_k} \circ \Gamma_k$ in \eqref{uniformbdofwidehatfk}, H\"older's inequality, the estimate of $\|h_k\|_{L^2(M,g_k)}$ in \eqref{pHkerror}, and \eqref{e4hatuk+lsbound} we can bound
\begin{align}\label{hatQ+lsbound}
\begin{aligned}
 \int_{ \widehat B_{R/4} (0)}\big| \widehat f_ke^{4\widehat{w}_k}\big|^{s_0} \dv_{\widehat g_k} 
 \leqslant &C\int_{\widehat B_{R/4} (0)}\big|\alpha_k\widetilde{f}_{\lambda_k} \circ \Gamma_k\big|^{s_0}e^{4s_0\widehat{w}_k}~\dv_{\widehat g_k} \\
 &+C\int_{\widehat B_{R/4}(0) }|\widetilde{h}_k\circ\Gamma_k|^{s_0}e^{2s_0\widehat{w}_k}e^{2s_0\widehat{w}_k}~\dv_{\widehat g_k}\\
 \leqslant & C_R\bigg(\int_{\widehat B_{R/4}(0) }e^{4s_1\widehat{w}_k}~\dv_{\widehat g_k}\bigg)^{s_0/s_1} \\
 &+C\bigg(\int_{\widehat B_{R/4}(0) }|\widetilde{h}_k\circ\Gamma_k|^2e^{4\widehat{w}_k}~\dv_{\widehat g_k}\bigg)^{s_0/2} \\
 & \times\bigg(\int_{\widehat B_{R/4}(0) }e^{4s_2\widehat{w}_k }~\dv_{\widehat g_k}\bigg)^{s_0/(2s_2)} \\
 \leqslant &C_R+C_R\|h_k\|_{L^2(M,g_k)}\bigg(\int_{\widehat B_{R/4}(0) }e^{4s_1\widehat{w}_k }~\dv_{\widehat g_k}\bigg)^{s_0s_1/(2s_2^2)}\\
\leqslant & C_R.
\end{aligned}
\end{align}
Plugging \eqref{hatQ+lsbound} into \eqref{eqBeforeLimitingEquation} gives 
$$\int_{ \widehat B_{R/4} (0)}|\PP_{\widehat g_k}\widehat w_k|^{s_0} \dv_{\widehat g_k}\leqslant C_R,$$
which together with \eqref{hatuk+lsbound} implies that $\widehat w_k$ is bounded in $W^{4,s_0}_{\rm loc}(\mathbf R^4)$. In particular, Sobolev embedding theorem implies that $\widehat{w}_k \to \widehat w_\infty$ strongly in $C^{0,\alpha}_{\rm loc}(\mathbf R^4)\cap H^2_{\rm loc}(\mathbf R^4)$ with $0<\alpha<1-1/s_0$. It remains to establish \eqref{volumeofhatWinftybd}. Indeed, by Fatou's lemma, \eqref{eqExponentialMap4} and the fact that $\int_Me^{4w}~\dvg=1$ we obtain
\begin{eqnarray*}\int_{\widehat B_R(0)}e^{4\widehat w_\infty} \dz&\leqslant&\liminf_{k\rightarrow+\infty}\int_{\widehat B_R(0)}e^{4\widehat w_k}~d\mu_{\widehat g_k}\nonumber\\
&\leqslant&\liminf_{k\rightarrow+\infty}\int_{B_{2Rr_k}(x_k)}e^{4w_k}~\dvg\leqslant1.
\end{eqnarray*}
Passing to the limit $R\rightarrow+\infty$ we find 
$$\int_{\mathbf R^4}e^{4\widehat w_\infty} \dz\leqslant1.$$
We thus finish the proof of Claim 3. \qed

\medskip
\noindent\textbf{Claim 4}. The assertions in Theorem \ref{main3}(ii) hold true.

\noindent{\itshape Proof of Claim 4}. 
Since $0<r_k/\sqrt{\lambda_k}\leqslant\sqrt{3/a_1}$ by \eqref{volumelowebd2}, we have two possibilities.
\smallskip

\noindent\underline{Case 1}. There holds $\limsup_{k \to +\infty}r_k/\sqrt{\lambda_k}=0$. In this scenario, recall that the estimate $d(x_k, x_0) = O(\sqrt{\lambda_k})_{k \nearrow +\infty}$ in \eqref{bdofdistxkandx0} implies that $|z_k| = O(\sqrt{\lambda_k})_{k \nearrow +\infty}$ if $k$ is large enough. This together with \eqref{positivelbandubofalphaklambdak} and \eqref{widehatfk} implies that there exists some $r_0 \leqslant 64\pi^2$ such that, up to a subsequence,
\begin{equation}\label{limitofhatfk1}
\lim_{k \to +\infty} \alpha_k\widetilde f_{\lambda_k}\big(\Gamma_k(z)\big)=r_0
\end{equation}
uniformly in any fixed ball $\widehat B_R (0)$. Next we derive the equation for the limit function $\widehat w_\infty$ in Claim 3. We multiply by a smooth function $\varphi$ with compact support on the both sides of equation \eqref{eqBeforeLimitingEquation} and then do integrating by parts to obtain
\[
\langle\PP_{\widehat g_k}\widehat w_k,\varphi\rangle
=\int_{\mathbf R^4}\alpha_k\widetilde f_{\lambda_k}\circ\Gamma_ke^{4\widehat w_k}\varphi~ \dv_{\widehat g_k}+\int_{\mathbf R^4}\widetilde h_{k}\circ\Gamma_ke^{4\widehat w_k}\varphi~ \dv_{\widehat g_k}.
\]
 By the fact that $\widehat g_k \to (dz)^2$ in $C^\infty_{\rm loc}(\mathbf R^4)$ and the estimate of $\|h_k\|_{L^2(M,g_k)}$ in \eqref{pHkerror} we send $k$ to infinity in the equality above to conclude that the function $\widehat w_\infty$ solves the equation
\begin{equation}\label{eqLimitingEquationW}
\Delta^2_z\widehat w_\infty=r_0e^{4\widehat w_\infty}
\end{equation}
in $\mathbf R^4$. 

Since in this case we can obtain a very precise form for $w_\infty$ from \eqref{eqLimitingEquationW}, we need more work by showing that $r_0>0$. Indeed, suppose that this is not true, then we are led to two cases: $r_0=0$ or $r_0<0$. When $r_0=0$, it follows from \eqref{eqLimitingEquationW} and \eqref{volumeofhatWinftybd} that the function $\widehat w_\infty$ solves
\[
\Delta^2_z\widehat w_\infty=0
\]
with the finite energy condition
\[
\int_{\mathbf R^4}e^{4\widehat w_\infty} \dz <+\infty.
\]
Now it follows from \cite[Theorem 3]{Mar09} that $\widehat w_\infty$ is a polynomial of order exactly two, which is also bounded in $\mathbf R^4$. Consequently, $\widehat w_\infty$ is at most linear. Therefore, we can make use of \cite[Theorem 2.4]{ARS} to conclude that 
\[
\Delta_z\widehat w_\infty\equiv c_0>0
\] 
everywhere in $\mathbf R^4$. Using this fact, on one hand, the strong convergence $\Delta _{\widehat g_k} \widehat w_k \to \Delta_z \widehat w_\infty$ in $L^2_{\rm loc}(\mathbf R^4)$ implies, for arbitrary but fixed $R>0$, that
\[
 \lim_{k\rightarrow+\infty}\int_{\widehat B_{R/2}(0)}|\Delta_{\widehat g_k}\widehat w_k| \dv_{\widehat g_k}=\int_{ \widehat B_{R/2} (0) }|\Delta_z\widehat w_\infty| \dz=\frac{c_0}4 \pi^2 \Big( \frac R2 \Big)^4.
\]
However, on the other hand, we can estimate
\begin{align*}
\int_{\widehat B_{R/2}(0)} & |\Delta_{\widehat g_k}\widehat w_k| \dv_{\widehat g_k}\\
&= r_k^{-2}\int_{\widehat B_{Rr_k}(z_k)}|\Delta_{\widetilde g_0} \widetilde w_k| \dv_{\widetilde g_0} \\
&\leqslant r_k^{-2}\int_{B_{2Rr_k}(x_k)}|\Delta_{g_0}w_k| \dvg \\
 &=r_k^{-2} \int_{B_{2Rr_k}(x_k)} \int_M\big|\Delta_{g_0} \mathbb G(x,y)\big|\big|\alpha_kf_{\lambda_k}(y)+h_k\big|e^{4w_k(y)} \dvg (y) \dvg (x) \\
 &\leqslant C r_k^{-2}\int_M\big|\alpha_kf_{\lambda_k}(y)+h_k\big|e^{4w_k(y)}\int_{B_{2Rr_k}(x_k)}d(x,y)^{-2} \dvg (x) \dvg (y) \\
 &\leqslant CR^2\int_M\big|\alpha_kf_{\lambda_k}+h_k\big|e^{4w_k} \dvg \\
&=O(R^2).
 \end{align*}
Putting these facts together, we eventually obtain
\begin{equation}\label{Deltazhatwinfty}
\frac{c_0}4 \pi^2 \Big( \frac R2 \Big)^4 =O(R^2),
\end{equation}
which is impossible if we let $R$ sufficiently large. We now rule out the case $r_0<0$. Indeed, in this scenario, we apply \cite[Theorem 2]{Mar08} to \eqref{eqLimitingEquationW} to get 
\[
\lim_{t \to +\infty}\Delta\widehat w_\infty (t \xi)=c_1\neq0
\]
uniformly in $\xi \in K$ where $K\subset \mathbb S^3$ is any compact set with positive Hausdorff measure. Then, for $k$ large enough, we have the following estimates similar to the ones leading to \eqref{Deltazhatwinfty}
 \begin{align*}
 C \Big( \frac R2 \Big)^4 &\leqslant \int_{\widehat B_{R/2}(0) \cap (\mathbf R^+K)}|\Delta_z\widehat w_\infty| \dz \\
& \leqslant\lim_{k\rightarrow+\infty} \int_{\widehat B_{R/2}(0)}|\Delta_{\widehat g_k}\widehat w_k| \dv_{\widehat g_k}\\
& =O(R^2),
 \end{align*}
which, again, is a contradiction if $R$ is sufficiently large. Hence, we have proved that $r_0>0$. Since $e^{4\widehat w_\infty}\in L^1(\mathbf R^4)$ by \eqref{volumeofhatWinftybd}, the well-known classification theorem in \cite{Lin} then implies that either there exists a constant $c_0>0$ such that
\[
-\Delta_z\widehat w_\infty\geqslant c_0
\]
everywhere in $\mathbf R^4$ or there exist some $\mu_0>0$ and $z_0\in\mathbf R^4$ such that
\begin{equation}\label{hatWinfty}
\widehat w_\infty(z)=\log\Big(\frac{2\mu_0}{1+\mu_0^2|z-z_0|^2}\Big)-\frac14\log\frac{r_0}{6}.
\end{equation}
We can rule out the first alternative in the same way as \eqref{Deltazhatwinfty}. Hence, the second alternative must occur. Now, recall by Claim 3 that we have the strong convergence $\widehat w_k \rightharpoonup \widehat w_\infty$ in $C^{0,\alpha}_{\rm loc}(\mathbf R^4)\cap H^2_{\rm loc}(\mathbf R^4)$ for some $0<\alpha<1$. This together with the decomposition 
\begin{align*}
\PP_{\widehat g_k}(\widehat w_k-\widehat w_\infty)+(\PP_{\widehat g_k}-\Delta_z^2)\widehat w_\infty =& \widetilde{h}_k\circ\Gamma_ke^{4\widehat w_k}
+\big(\alpha_k\widetilde{f}_{\lambda_k}\circ\Gamma_k-r_0\big)e^{4\widehat w_k}\\
&+r_0(e^{4\widehat w_k}-e^{4\widehat w_\infty}),
\end{align*}
\eqref{pHkerror} and \eqref{limitofhatfk1} implies that $\widehat w_k \to \widehat w_\infty$ strongly in $H^4_{\rm loc}(\mathbf R^4)$. 

Up to this point, we are ready to estimate the number of blow-up points. Recall that we have already had $I \leqslant 8$, however, in the present case, we aim to show that indeed $I \leqslant 4$. Clearly at each blow-up point, say $x_0$ as before with the same notations used up to this position for simplicity, from the explicit formula \eqref{hatWinfty} we can compute
\[
 \int_{\mathbf R^4}e^{4\widehat w_\infty} \dz =\frac{6}{r_0}\int_{\mathbf R^4}\bigg(\frac{2\mu_0}{1+\mu_0^2|z-z_0|^2}\bigg)^4 \dz=\frac{16\pi^2}{r_0}.
\]
Since $e^{\widehat w_k} \to e^{4\widehat w_\infty}$ strongly in $L^1_{\rm loc}(\mathbf R^4)$ as $k \to +\infty$ and $r_0\leqslant64\pi^2$, we have for $R$ and $k$ sufficiently large
\begin{align*}
\frac{15}{64} &\leqslant \int_{\widehat B_{R}(0)} e^{4\widehat w_k} \dv_{\widehat g_k}\leqslant\int_{B_{2Rr_k}(x_k)}e^{4w_k}~\dvg.
\end{align*}
Since the number of blow-up points is finite, if we choose $k$ even larger, we deduce that the sets $B_{2Rr_k}(x_k)\subset B_\delta(x_0)$ and they are non-overlap at different blow-up points. Keep in my that $\int_M e^{4w_k } \dvg =1$. From this we deduce that the number of blow-up points must less than or equal to $4$, namely $I\leqslant 4$. 

Finally, we notice that, up to a translation and a scaling, $\widehat w_\infty$ has the form
$$\widehat w_\infty(z)=\log\Big(\frac{4\sqrt{6}}{4\sqrt{6}+|z|^2}\Big).$$
Indeed, it suffices to substituting \eqref{hatWinfty} into the expression
$$\widehat w^*_\infty(z):=\widehat w_\infty\Big(e^{-\widehat w_\infty(z_0)}r_0^{-1/4}z+z_0\Big)-\widehat w_\infty(z_0)-\frac14\log \frac{r_0}{6} .$$
We thus obtain the alternative (ii)(a) in Theorem \ref{main3}.

\medskip
\noindent\underline{Case 2}. We now suppose that $\limsup_{k \to +\infty}r_k/\sqrt{\lambda_k}>0$. Since $r_k/\sqrt{\lambda_k}$ is bounded from above and $|z_k|=O(\sqrt{\lambda_k})_{k\nearrow+\infty}$, we may assume that
\[
\limsup_{k \to +\infty}\frac{r_k}{\sqrt{\lambda_k}}=d_0>0
\]
and that
\[
\limsup_{k \to +\infty}\frac{z_k}{\sqrt{\lambda_k}}=\vec{c}_0
\]
for some constant vector $\vec{c}_0$. This together with \eqref{positivelbandubofalphaklambdak} and \eqref{widehatfk} implies that there exists some constant $r_0$ with $8\pi^2\leqslant r_0\leqslant 64\pi^2$ such that
\begin{equation*}\label{limitofhatfk}
 \limsup_{k \to +\infty}\widehat f_k(z)=r_0\Big(1+\frac12 \Hess_{f_0}(x_0)\big[\vec{c}_0+d_0z,\vec{c}_0+d_0z\big]\Big)
\end{equation*}
uniformly in $\widehat{B}_R(0)$. Arguing the same way as in the proof of Case 1 to obtain \eqref{eqLimitingEquationW}, the limiting function $\widehat w_\infty$ solves the equation
\begin{equation}\label{eqLimitingEquationW-slow}
\Delta^2_z\widehat w_\infty=r_0\Big(1+\frac12 \Hess_{f_0}(x_0)\big[\vec{c}_0+d_0z,\vec{c}_0+d_0z\big]\Big)e^{4\widehat w_\infty}
\end{equation}
in $\mathbf R^4$. Furthermore, in view of \eqref{volumeofhatWinftybd} and the $L^1$-bound \eqref{bdoftotalQcurvatureerror}, we have
\[
\int_{\mathbf R^4}e^{4\widehat w_\infty} \dz<+\infty
\]
and
\[
\int_{\mathbf R^4}\Big|1+\frac12\Hess_{f_0}(x_0)[\vec{c}_0+d_0z,\vec{c_0}+d_0z]\Big|e^{4\widehat w^*_\infty(z)} \dz <+\infty.
\]
By denoting
$$F_\infty=r_0\Big(1+\frac12\Hess_{f_0}(x_0)[\vec{c}_0+d_0z,\vec{c_0}+d_0z]\Big),$$
it follows from the decomposition
\begin{align*}
\PP_{\widehat g_k}(\widehat w_k-\widehat w_\infty)+(\PP_{\widehat g_k}-\Delta_z^2)\widehat w_\infty =&\big(\alpha_k\widetilde{f}_{\lambda_k}\circ\Gamma_k + \widetilde{h}_k\circ\Gamma_k -F_\infty \big) e^{4\widehat w_k(z)}\\
 &+F_\infty(e^{4\widehat w_k}-e^{4\widehat w_\infty})
\end{align*}
and \eqref{pHkerror} that $\widehat w_k \to \widehat w_\infty$ strongly in $H^4_{\rm loc}(\mathbf R^4)$ as $k \to +\infty$. 

Finally, by performing a translation and a scaling, equation \eqref{eqLimitingEquationW-slow} can be reduced as
$$\Delta^2_z\widehat w_\infty=\Big(1+\frac12 \Hess_{f_0}(x_0)\big[z,z\big]\Big)e^{4\widehat w_\infty}.$$
We thus obtain the alternative (ii)(b) in Theorem \ref{main3}. \qed

The proof of Theorem \ref{main3} is complete.
\end{proof}


\subsubsection{Degenerate case}

Now we consider the degenerate case. An analogue of Theorem \ref{main3} is the following result.

\begin{theorem}\label{main4}
Assume all the conditions, expcept for the assumption of the non-degeneracy of the function $f_0$ at some maxima, in Theorem \ref{main3} above. If, in addition, $(M,g_0)$ is locally conformally flat and $f_0$ satisfies the {\bf Condition A} with $d_0, A_0>0$, then for $w_k$ defined as in the Theorem \ref{main3} there exist suitable $I\in\mathbb{N}$ with $I \leqslant 8$, $r_k^{(i)}\searrow0$ and $x_k^{(i)} \to x_\infty^{(i)}\in M$ with $f_0(x_\infty^{(i)})=0$, $1\leqslant i\leqslant I$ such that the following hold
\begin{enumerate}[label=\rm (\roman*)]
 \item $w_k \to -\infty$ locally uniformly on
\[
M_\infty=M\backslash\{x_\infty^{(i)}:1\leqslant i\leqslant I\}.
\]
 \item For each $1\leqslant i\leqslant I$, we have
\[
\widehat w_k(z):=
\widetilde{w}_k\big(z_k^{(i)}+r_k^{(i)}z\big)+\log r_k^{(i)} \to \widehat w_\infty(z)
\]
strongly in $H_{\rm loc}^4(\mathbf R^4)$, where $z_k^{(i)}=\exp_{x_\infty^{(i)}}^{-1}(x_k^{(i)})$ and $\widehat w_\infty$ induces a metric
\[
g_\infty=e^{4\widehat w_\infty}g_{\mathbf R^4}
\]
on $\mathbf R^4$ of locally bounded curvature and of volume less than or equal $1$.
\end{enumerate}
\end{theorem}

\begin{proof}
For simplicity and clarity, we still use the notations in the proof of Theorem \ref{main3}. We first notice that Lemma \ref{concentration0} and the upper bound for $\int_M |Q_{g_k}| \dv_{g_k}$ as in \eqref{bdoftotalQcurvatureerror} continue to hold even if $f_0(x)$ has a degenerate maxima. Consequently, the bounds for $\alpha_k\lambda_k$ as in \eqref{positivelbandubofalphaklambdak} also holds as well.
 
 \medskip
 \noindent\textbf{PART 1}. The proof of statement (i) in Theorem \ref{main4} is then identical with that of the corresponding statement in Theorem \ref{main3}.

 \medskip
 
 \noindent\textbf{PART 2}. We now examine the blow-up behavior of $w_k$ near the blow-up point $x_0$. Since $(M,g_0)$ is locally comformally flat, we may assume that $M$ is flat around $x_0$, namely
\[
(g_0)_{ij}=\delta_{ij}
\]
in $B_\delta(x_0)$ for some fixed but small $\delta>0$. 

\medskip
\noindent\textbf{Claim 1}. There is a constant $\rho_0>0$ such that for each $\rho\in(0,\rho_0)$ to be determined later, there exists a sequence of positive numbers $(r_k)_k$ and a sequence of points $(x_k)_k\subset\overline{B_\delta(x_0)}$ satisfying
\begin{subequations}\label{volumeconcentration4}
\begin{empheq}[left=\empheqlbrace]{align}
\label{r_klambda_k-Degenerate}
&\lim_{k \to +\infty}r_k=0,\\
\label{rho_k-Degenerate}
&\int_{B_{r_k}(x_k)}e^{4w_k} \dvg =\rho,\\
\label{x_k->w_k->-Degenerate} 
&x_k \rightarrow x_0\quad\mbox{and}\quad w_k(x_k)\rightarrow+\infty ~~\mbox{as}~k\rightarrow\infty,\\
\label{integralw_k<=rho-Degenerate}
&\int_{B_{r_k}(y)}e^{4w_k} \dvg \leqslant\rho \quad \text{ for all } \quad {y\in B_{\sqrt{r_k}}(x_k)}.
\end{empheq}
\end{subequations}

\noindent{\itshape Proof of Claim 1}. The proof of \eqref{volumeconcentration4} is essentially similar to the proof of \eqref{volumelowebd2}. Notice that in the degenerate case we cannot assert an upper bound for $r_k/\sqrt{\lambda_k}$ as shown in \eqref{r_klambda_k}. However, we still have the estimate $r_k = o(1)_{k \nearrow +\infty}$ shown in \eqref{r_klambda_k-Degenerate}. To realize this, we first notice by Lemma \ref{concentration0}, the estimate $Q_{g_k}^+\leqslant\alpha_k\lambda_k+|h_k|$, \eqref{positivelbandubofalphaklambdak} and \eqref{pHkerror} that
\begin{equation}
 \label{volumelowebd3}
 8\pi^2-o(1)_{k \nearrow +\infty}\leqslant\int_{B_r(x_0)}Q_{g_k}^+e^{4w_k}~\dvg\leqslant (64\pi^2+o(1)_{k \nearrow +\infty})\int_{B_r(x_0)}e^{4w_k}~\dvg
\end{equation}
for all $r>0$. In particular, we have 
\[
\int_{B_\delta(x_0)}e^{4w_k}~\dvg\geqslant\frac{7}{65}=:\rho_0
\]
for $k$ large. With the help of the above estimate, we can follow the proof of Claim 1 in Theorem \ref{main3} to obtain that for each $\rho\in(0,\rho_0)$ there exists $r_k\in(0,\delta)$ such that
$$\sup_{x\in\overline{B_\delta(x_0)}}\int_{B_{r_k}(x)}e^{4w_k}~\dvg=\rho.$$
This implies that
\begin{equation}
 \label{volumeupperbd}
 \int_{B_{r_k}(x_0)}e^{4w_k}~\dvg\leqslant\rho.
\end{equation}
We are now ready to conclude \eqref{r_klambda_k-Degenerate}. Indeed, by the way of contradiction and up to a subsequence, we may assume that $\lim_{k\rightarrow+\infty}r_k=r_0>0$. Then, there holds $B_{r_0/2}(x_0)\subset B_{r_k}(x_0)$ provided $k$ is sufficiently large. This together with \eqref{volumeupperbd} and the choice of $\rho$
 yields 
 $$\int_{B_{r_0/2}(x_0)}e^{4w_k}~\dvg\leqslant\rho<\frac{7}{65}.$$
But this contradicts with \eqref{volumelowebd3} and the proof of \eqref{r_klambda_k-Degenerate} is complete. As for the other assertions in \eqref{volumeconcentration4}, their proofs are identical with those of Claim 1 in Theorem \ref{main3}.
\qed

Lacking of a bound for $r_k/\sqrt{\lambda_k}$ brings us difficulty to obtain a local bound for $ \alpha_k\widetilde f_{\lambda_k} \circ\Gamma_k$ as in \eqref{uniformbdofwidehatfk}. However, under an additional hypothesis on the flatness of $(M, g_0)$ we can proceed with some tools developed in \cite{Mar09} together with \textbf{Condition A} to regain its local boundedness; see Claim 3 below.

Now, since $(M,g_0)$ is locally comformally flat, we may assume that $M$ is flat around $x_0$, namely
\[
(g_0)_{ij}=\delta_{ij}
\]
in $B_\delta(x_0)$ for some fixed but small $\delta>0$. On one hand, this helps us to conclude that
\begin{equation*}\label{PullbackOfVolumeForm}
\dv_{\widetilde g_0} = \exp_{x_0}^* (\dvg) = \dz.
\end{equation*}
This and the relation $\widehat g_k = r_k^{-2} \Gamma_k^* \widetilde g_0$ imply that
\[
\dv_{\widehat g_k} = r_k^{-4}\dv_{ \Gamma_k^* \widetilde g_0} = \dz.
\]
On the other hand, the Paneitz operator becomes the bi-Laplace operator in $B_\delta(x_0)$. The equation \eqref{pQeerror} becomes 
\[
\Delta^2w_k=\alpha_kf_{\lambda_k}e^{4w_k}+h_ke^{4w_k}
\]
in $B_\delta(x_0)$. Let $\widehat w_k$ be defined as in \eqref{eqWHat_k}, then $\widehat w_k$ solves
\begin{equation}\label{ConformlFlatEq}
 \Delta^2\widehat w_k=\widehat f_ke^{4\widehat w_k}
\end{equation}
in $\widehat D_{k,\delta}$, where, as before,
\[
 \widehat f_k=\alpha_k\widetilde{f}_{\lambda_k} \circ \Gamma_k +\widetilde{h}_k\circ\Gamma_k
\]
and
\[
\widehat D_{k,\delta} := \{ z \in \mathbf R^4 : |z_k + r_k z| < \delta \}. 
\]
Also because $\widehat B_{1/2} (0) \subset (\exp_{x_0}\circ\Gamma_k)^{-1}(B_{r_k}(x_k)) \subset \widehat B_2 (0)$ it follows from \eqref{rho_k-Degenerate} that
\begin{equation}\label{volumeconcentration5}
\int_{\widehat B_{1/2} (0)}e^{4\widehat w_k} \dz
 \leqslant \rho \leqslant 
 \int_{\widehat B_2 (0)}e^{4\widehat w_k} \dz
 \end{equation}
and, similar to \eqref{volumeconcentration3}, we rewrite \eqref{integralw_k<=rho-Degenerate} to get
\begin{equation}\label{volumeconcentration6}
 \int_{\widehat B_{1/2}(z)}e^{4\widehat w_k} \dz \leqslant\rho,
\end{equation}
for all $z\in\widehat B_{1/(2\sqrt{r_k})}(0)$. 

\medskip
\noindent\textbf{Claim 2}. The sequence $\widehat w_k$ is bounded in $W^{3,s}_{\rm loc} (\mathbf R^4)$ for any $1< s < 4/3$.

\noindent{\itshape Proof of Claim 2}. Fix any $R>8$, we let $\widehat w_k^{(\pm)}$ solve
\begin{equation}\label{wkhatpm}
\left\{
\begin{aligned}
 \Delta^2\widehat w_k^{(\pm)}&=(\Delta^2\widehat w_k)^{\pm} &\text{ in } & \widehat B_R(0) ,\\ 
 \widehat w_k^{(\pm)}&=0 &\text{ on } &\partial\widehat B_R (0),\\
 \Delta\widehat w_k^{(\pm)}& =0 &\text{ on } &\partial\widehat B_R (0).
\end{aligned}
\right.
\end{equation}
Using the maximum principle twice, we obtain $\widehat w_k^{(+)}\geqslant 0\geqslant \widehat w_k^{(-)}$. In addition, $\widehat w_k$ can be decomposed as 
\begin{equation}\label{wkhatdecomp}
\widehat w_k=\widehat w_k^{(+)}+\widehat w_k^{(-)}+\widehat w_k^{(0)},
\end{equation}
where $\widehat w_k^{(0)}$ solves
\[
\left\{
\begin{aligned}
 \Delta^2\widehat w_k^{(0)} & =0 &\text{ in }& \widehat B_R (0),\\ 
 \widehat w_k^{(0)}& =\widehat w_k &\text{ on }& \partial\widehat B_R (0),\\
 \Delta\widehat w_k^{(0)}&=\Delta\widehat w_k &\text{ on }& \partial\widehat B_R (0).
\end{aligned}
\right.
\]
The next goal is to show the boundedness of $\widehat w_k^{(+)}$ 
in $W^{4,s_0}_{\rm loc}(\mathbf R^4)$
for some $s_0>1$. To see this, we observe, by \eqref{positivelbandubofalphaklambdak}, \eqref{volumeconcentration6} and \eqref{pHkerror}, the bound
\begin{align*}
 \int_{\widehat B_r(z)}\big(\Delta^2\widehat w_k\big)^+ \dz
 &= \int_{\widehat B_r(z)}\widehat f_k^+e^{4\widehat w_k} \dz\\
 &\leqslant \int_{\widehat B_r(z)} \alpha_k\big(\widetilde f_k\circ\Gamma_k\big)^+e^{4\widehat w_k} \dz +\int_{\widehat B_r(z)}\big|\widetilde h_k\circ\Gamma_k\big|e^{4\widehat w_k} \dz\\
 &\leqslant 65\pi^2\rho+o(1)_{k \nearrow +\infty}
\end{align*}
for all $z\in\widehat B_{R/2} (0)$ and for $r>0$ small. Hence, by choosing $\rho$ sufficiently small we obtain the bound
\[
\int_{\widehat B_r(z)}\big(\Delta^2\widehat w_k\big)^+\dz <8\pi^2
\]
for all $z\in\widehat B_{R/2} (0)$ and for $r>0$ small. In view of the equation \eqref{wkhatpm} satisfied by $\widehat w_k^{(\pm)}$, we can apply \cite[Lemma 2.3]{Lin} and a finite covering argument to find a positive constant $s_1>1$ such that
\begin{equation}
 \label{ewkhatpL4s1bd}
 \int_{\widehat B_{R/2} (0)}e^{4s_1\widehat w_k^{(+)}} \dz \leqslant C_R.
\end{equation}
Keep in mind that $\widehat f_k^+ \leqslant \alpha_k \lambda_k + |\widetilde{h}_k\circ\Gamma_k|$. Hence, by repeating an argument used in \eqref{hatQ+lsbound} together with \eqref{ewkhatpL4s1bd} we can find some $1<s_0<\min\{4/3,s_1\}$ such that
\begin{align*}
\int_{\widehat B_{R/2} (0)} &\big(\widehat f_k^+e^{4\widehat w_k}\big)^{s_0} \dz \\
&\leqslant C\int_{\widehat B_R (0)}(\alpha_k\lambda_k)^{s_0}e^{4s_0\widehat w_k} \dz +C\int_{\widehat B_{R}(0)}|\widetilde{h}_k\circ\Gamma_k|^{s_0}e^{4s_0\widehat{w}_k} \dz\leqslant C_R.
\end{align*}
Plugging the estimate above into \eqref{wkhatpm} gives
\[
\int_{\widehat B_{R/2}(0)}|\Delta^2\widehat w_k^{(+)}|^{s_0} \dz \leqslant C_R.
\]
This together with Sobolev's inequality implies that $\widehat w_k^{(+)}$ is bounded in $W^{4,s_0}_{\rm loc}(\mathbf R^4)$. Therefore, $\widehat w_k^{(+)}$ is bounded in $C^{0,\alpha}_{\rm loc}(\mathbf R^4)$ for some $0<\alpha\leqslant1-1/s_0$. Moreover, we let $\gamma=1/17$. It then follows from \eqref{bdoftotalQcurvatureerror} that
\begin{align*}
 \gamma\int_{\widehat B_R (0)}|\Delta^2\widehat w_k| \dz
 &\leqslant \gamma\int_{\widehat B_R (0)}\big|\widehat f_k\big|e^{4\widehat w_k} \dz
\leqslant \gamma\int_M|Q_{g_k}| \dvg <8\pi^2.
\end{align*}
Then repeating the previous argument we have
\begin{equation}
 \label{ewidehatwkmL4s1bd}
 \int_{\widehat B_{R/2}(0)}e^{\pm4s_1\gamma\widehat w_k^{(-)}} \dz \leqslant C_R.
\end{equation}
Also, there holds
\begin{equation}
\label{widehatwkpmw3sbd}
\big\|\widehat w_k^{(\pm)}\big\|_{W^{3,s}(\widehat B_R(0))}\leqslant C_R
\end{equation}
for all $s\in[1,4/3)$. Now, it follows from Jensen's inequality, the decomposition of $\widehat w_k$ in \eqref{wkhatdecomp}, H\"older's inequality, and \eqref{ewidehatwkmL4s1bd} that
\begin{equation}\label{wkhatzerointupbd}
\begin{aligned}
\exp\Big(\dashint_{\widehat B_r(z)}\widehat w_k^{(0)} \dz \Big)
&\leqslant \bigg(\dashint_{\widehat B_r(z)} e^{\frac{4s_1\gamma}{1+s_1\gamma}\widehat w_k^{(0)}} \dz \bigg)^\frac{1+s_1\gamma}{4s_1\gamma} \\
&\leqslant \bigg(\dashint_{\widehat B_r(z)} e^{\frac{4s_1\gamma}{1+s_1\gamma}(\widehat w_k-\widehat w_k^{(-)})} \dz \bigg)^\frac{1+s_1\gamma}{4s_1\gamma} \\
&\leqslant \bigg(\dashint_{\widehat B_r(z)} e^{4\widehat w_k} \dz \bigg)^\frac{(1+s_1\gamma)^2}{4(s_1\gamma)^2}\bigg(\dashint_{\widehat B_r(z)} e^{-4s_1\gamma\widehat w_k^{(-)}} \dz \bigg)^\frac{1}{4s_1\gamma} \\
&\leqslant C_R
\end{aligned}
\end{equation}
for all $z\in\widehat B_{R/4}(0)$ and for $r>0$ small. In \eqref{wkhatzerointupbd}, the symbol $\dashint_\Omega h $ denotes the average of $h$ over $\Omega$. Notice that the estimate \eqref{hatwkw3sbound} also holds in the current case. This together with \eqref{widehatwkpmw3sbd} implies that
\begin{align*}
\|\Delta\widehat w_k^{(0)}\|_{L^1(\widehat B_R(0))}&\leqslant \|\Delta\widehat w_k\|_{L^1(\widehat B_R(0))}+\|\Delta\widehat w_k^{(+)}\|_{L^1(\widehat B_R(0))} +\|\Delta\widehat w_k^{(-)}\|_{L^1(\widehat B_R(0))}\leqslant C_R.
\end{align*}
Since $\Delta(\Delta \widehat w_k^{(0)})=0$, we can apply \cite[Proposition 11]{Mar09} to get 
\begin{equation}\label{Deltawidehatwk0bd}
 \|\Delta\widehat w_k^{(0)}\|_{C^l(\widehat B_{R/2}(0))}\leqslant C_R(l)
\end{equation}
for every $l\in\mathbb N$. Notice that by the mean value property for biharmonic functions, see \cite[Lemma 2.2]{ARS}, we have
\[
\widehat w_k^{(0)}(z) = \dashint_{\widehat B_r(z)}\widehat w_k^{(0)} \dz
+\frac{r^2}{12}\Delta\widehat w_k^{(0)}(z).
\]
This together with \eqref{wkhatzerointupbd} and \eqref{Deltawidehatwk0bd} implies that
\[
\widehat w_k^{(0)}(z)\leqslant C_R
\]
for all $z\in\widehat B_{R/4} (0)$. In view of \eqref{Deltawidehatwk0bd}, we may apply weak Hanack inequality, see \cite[Theorem 8.18]{GT98}, to the function $C_R-\widehat w_k^{(0)}$ to obtain that 
\begin{itemize}
 \item either $\widehat w_k^{(0)}$ uniformly converges to $-\infty$ on $\widehat B_{R/4} (0)$
 \item or $\|\widehat w_k^{(0)}\|_{L^1(\widehat B_{R/2} (0))}\leqslant C_R$. 
\end{itemize} 
If the first case occurs, then from the decomposition of $\widehat w_k$ in \eqref{wkhatdecomp} and the boundedness of $\widehat w_k^{(+)}$ in $C^{0,\alpha}_{\rm loc}(\mathbf R^4)$, we know that
\[
\widehat w_k\leqslant C_R+w_k^{(0)},
\] 
which immediately implies that $\widehat w_k \to -\infty$ uniformly on $\widehat B_{R/4} (0)$ as $k \to \infty$. From this we deduce that
\[
\int_{\widehat B_{R/4} (0)}e^{4\widehat w_k} \dz \to 0
\]
as $k \to +\infty$, which contradicts \eqref{volumeconcentration5} since we have chosen $R>8$. Hence, we must have
\[
\|\widehat w_k^{(0)}\|_{L^1(\widehat B_{R/2} (0))}\leqslant C_R.
\] 
We then apply \cite[Proposition 11]{Mar09} again to get 
\[
 \|\widehat w_k^{(0)}\|_{C^l(\widehat B_{R/4} (0))}\leqslant C_R(l)
\]
for every $l\in\mathbb{N}$ and
\[
 \|\widehat w_k^{(0)}\|_{W^{3,s}(\widehat B_{R/4} (0))}\leqslant C_R(s)
\]
for every $1<s<4/3$. Clearly, the estimate of $\widehat w_k^{(0)}$ in $C^l (\widehat B_{R/4}(0))$ above together with the decomposition of $\widehat w_k$ in \eqref{wkhatdecomp} implies that
\begin{equation}\label{UpperBoundWidehatW_k}
\widehat w_k\leqslant C_R
\end{equation}
in $\widehat B_{R/4}(0)$. Moreover, the estimate of $\widehat w_k^{(0)}$ in $W^{3,s}(\widehat B_{R/4} (0))$ together with \eqref{widehatwkpmw3sbd} tells us that $\widehat w_k$ is bounded in $W^{3,s}(\widehat B_{R/4} (0))$. Since $R$ is arbitrary, the sequence $\widehat w_k$ is bounded in $W^{3,s}_{\rm loc}(\mathbf R^4)$ for any $1<s<4/3$. \qed 

From Claim 2, up to a subsequence, there holds
 \[
 \widehat w_k\rightharpoonup\widehat w_\infty
 \]
 weakly in $W^{3,s}_{\rm loc}(\mathbf R^4)$ for some $1<s<4/3$ and almost everywhere on $\mathbf R^4$. By Fatou's lemma and \eqref{eqExponentialMap2}, we can deduce that
 \begin{align*}
 \int_{\widehat B_{R/2}(0)}e^{4\widehat w_\infty} \dz
& \leqslant \liminf_{k \to \infty}\int_{\widehat B_{R/2}(0)}e^{4\widehat w_k} \dz \\
& = \liminf_{k \to \infty}\int_{\widehat B_{Rr_k/2}(z_k)}e^{4\widetilde w_k} \dz \\
& \leqslant \liminf_{k \to \infty}\int_{B_{Rr_k}(x_k)} e^{4w_k } \dvg\\
&\leqslant1 .
 \end{align*}
 Passing to the limit as $R \to +\infty$ we find that $e^{4\widehat w_\infty}\in L^1(\mathbf R^4)$ with
 \begin{equation*}
 \int_{\mathbf R^4}e^{\widehat w_\infty} \dz=\lim_{R \to +\infty}\int_{\widehat B_{R/2} (0)}e^{4\widehat w_\infty} \dz\leqslant1.
 \end{equation*}
Now, recall $ \alpha_k\widetilde{f}_{\lambda_k} \circ \Gamma_k =\alpha_k\lambda_k+\alpha_k\widetilde{f}_0 \circ \Gamma_k$ and for simplicity, we denote
 \[
 \widehat f_{0k} = \alpha_k\widetilde{f}_0 \circ \Gamma_k ,
 \]
 which is non-positive. By \eqref{positivelbandubofalphaklambdak}, we may assume, up to a subsequence, that 
 \[
 \alpha_k\lambda_k \to \mu \in [8\pi^2, 64 \pi^2]
 \]
 as $k \to +\infty$. 
 
 \medskip
\noindent\textbf{Claim 3}. The sequence $ \alpha_k\widetilde{f}_0 \circ \Gamma_k$ is locally bounded (from below). 

\noindent{\itshape Proof of Claim 3}. Suppose that for some sequence $y_k \to y_0$ in $\mathbf R^4$ there holds
\[
\alpha_k|\widetilde{f}_0(\bar{z}_k)| \to +\infty
\]
as $k \to +\infty$, where
\[
\bar{z}_k=\Gamma_k (y_k) = z_k+r_ky_k.
\] 
Denote $p_k=\exp_{x_0}(\bar{z}_k)$. Because $\bar z_k \to 0$ as $k \to +\infty$, we then have $p_k \to x_0 \in M_0$ as $k \to +\infty$. From this we may assume from the beginning that $d(p_k)<d_0$. By Condition A, there exist some $A_0>0$ and a sequence of cones $K_{p_k}$ with vertex at $p_k$ such that
 \begin{equation}\label{inffk}
 \begin{split}
 A_0\inf_{y\in\widetilde{K}_{p_k}}\big| \alpha_k\widetilde{f}_0 (\Gamma_k (y))\big| & =A_0\alpha_k\inf_{z\in K_{p_k}}\big|\widetilde{f}_0(z) \big|
 \geqslant \alpha_k\big|\widetilde{f}_0(\bar{z}_k)\big| \to +\infty,
 \end{split}
 \end{equation}
where with a suitable labeling of coordinates 
 \begin{align*}
 \widetilde{K}_{p_k} &=\Gamma_k^{-1} (K_{p_k}) = \big\{z :z_k+r_k z \in K_{p_k}\big\}\\
 &=\Big\{y=(y^1,...,y^4) : \sqrt{\sum_{i=1}^3(y^i-y_k^i)^2}<y^4-y_k^4, \quad |y-y_k|<d_0/r_k\Big\}.
 \end{align*}
On the other hand, by the estimate $\alpha_k|\widetilde{f}_0 \circ \Gamma_k| = \alpha_k\lambda_k-\alpha_k \widetilde f_{\lambda_k} \circ \Gamma_k$ and the fact that $\int_Mf_{\lambda_k}e^{4u_k} \dvg =0$, as routine we can apply Fatou's lemma to get that
 \begin{equation}\label{fkupbd}
 \begin{split}
 \int_{\widehat B_{R/2}(0)}\liminf_{k \to +\infty} \big( \alpha_k|\widetilde{f}_0 \circ \Gamma_k| \big) e^{4\widehat w_\infty} \dz 
 & \leqslant \liminf_{k \to +\infty}\int_{\widehat B_{R/2} (0)} \alpha_k|\widetilde{f}_0 \circ \Gamma_k| e^{4\widehat w_\infty} \dz \\
 & \leqslant \liminf_{k \to +\infty}\int_{\widehat B_{Rr_k/2} (z_k)} (\alpha_k\lambda_k-\alpha_k \widetilde f_{\lambda_k})e^{4\widetilde w_k} \dz \\
 & = \liminf_{k \to +\infty}\int_{B_{Rr_k} (x_k)} (\alpha_k\lambda_k-\alpha_k f_{\lambda_k})e^{4w_k} \dvg \\
 &\leqslant\mu.
 \end{split}
 \end{equation}
 We thus obtain the contradiction from \eqref{inffk} and \eqref{fkupbd}, namely, the sequence $\widehat f_{0k}$ is locally bounded. \qed
 
We now make use of Claim 3 together with the local upper bound of $\widehat w_k$ in \eqref{UpperBoundWidehatW_k} to ensure, up to a subsequence, that
\[
\alpha_k ( \widetilde{f}_0 \circ \Gamma_k ) e^{4\widehat w_k} 
\stackrel{\ast}{\rightharpoonup} \widehat f_\infty e^{4\widehat w_\infty}
\]
weakly-* in the sense of measures, where $\widehat f_\infty\leqslant0$ is locally bounded from below. By setting
\[
F_\infty= \mu+\widehat f_\infty 
\]
and recall the definition of $\widehat f_k$ and $\mu$ we know that
\[
\widehat f_k e^{4\widehat w_k}
\stackrel{\ast}{\rightharpoonup} F_\infty e^{4\widehat w_\infty} 
\]
weakly-* in the sense of measures. Furthermore, we get from \eqref{bdoftotalQcurvatureerror} the following bound
 \[
 \int_{\mathbf R^4}|F_\infty| e^{4\widehat{u}_\infty} \dz \leqslant 2\mu\leqslant128\pi^2.
 \] 

\medskip
\noindent\textbf{Claim 4}. The sequence $\widehat w_k$ is bounded in $H^{4}_{\rm loc} (\mathbf R^4)$.

\noindent{\itshape Proof of Claim 4}. By repeating the estimate in \eqref{integralboundforhatQ+}, it follows from the local boundedness of $\alpha_k \widetilde f_{\lambda_k} \circ \Gamma_k$ established in Claim 3, the local upper boundedness of $\widehat w_k$ in \eqref{UpperBoundWidehatW_k}, and the smallness of $\|h_k\|_{L^2(M,g_k)}$ in \eqref{pHkerror} that
 \begin{align*}
 \int_{ \widehat B_{R/4} (0)}\big|\widehat f_ke^{4\widehat{w}_k}\big|^2 \dz &\leqslant C\int_{\widehat B_{R/4}(0)}\big|\alpha_k\widetilde{f}_{\lambda_k} \circ \Gamma_k\big|^2e^{8\widehat{w}_k} \dz+C\int_{\widehat B_{R/2} (0)}|\widetilde{h}_k\circ\Gamma_k|^2e^{8\widehat{w}_k} \dz \\
 &\leqslant C_R+C_R\int_{\widehat B_{R/2} (0)}|\widetilde{h}_k\circ\Gamma_k|^2e^{4\widehat{w}_k} \dz \\
&\leqslant C_R+C_R\|h_k\|^2_{L^2(M,g_k)}\leqslant C_R.
\end{align*}
This together with the equation satisfied by $\widehat w_k$ in \eqref{ConformlFlatEq} implies that $\widehat w_k$ is bounded in $H^{4}_{\rm loc}(\mathbf R^4)$. \qed

In view of Claim 4, we have that
\[
 \widehat w_k \rightharpoonup \widehat w_\infty
 \]
 weakly in $H^4_{\rm loc}(\mathbf R^4)$ and strongly in $C^{0,\alpha}_{\rm loc}(\mathbf R^4)$ for some $0<\alpha<1/2$. Moreover, by passing to the limit we deduce that $\widehat w_\infty$ solves the equation
 \[
 \Delta^2\widehat w_\infty=F_\infty e^{4\widehat{u}_\infty}
 \]
in $\mathbf R^4$. Finally, it follows from the decomposition
\begin{align*}
\Delta^2\widehat w_k-\Delta^2\widehat w_\infty&
= \big(\alpha_k\widetilde{f}_{\lambda_k}\circ\Gamma_k + \widetilde{h}_k\circ\Gamma_k -F_\infty\big)e^{4\widehat w_k} 
+F_\infty(e^{4\widehat w_k}-e^{4\widehat w_\infty})
\end{align*}
that $\widehat w_k \to \widehat w_\infty$ strongly in $H^4_{\rm loc}(\mathbf R^4)$.

\end{proof}


\section{Bubbling along the flow}

As in the case of Gaussian curvature flow studied by Struwe, it is unreasonable to expect that Theorem also holds for non-minimizing critical points 
\subsection{Bounds for total curvature along the flow}

Bounds analogue to Lemma \ref{bdoftotalQcurvature} can also be obtained for the solutions to the prescribed $Q$-curvature flow \eqref{eeforu} for $f_\lambda$. 

As in the static case, let $f_0 \leqslant 0$ be a smooth, non-constant function with $\max_M f_0 =0$. Let $0<\lambda<\lambda_0$ and let $f_\lambda = f_0 + \lambda$ as above where $\lambda_0 > 0$ is chosen in such a way that $f_{\lambda_0}$ changes sign and satisfies \eqref{eqKWTotalIntegralIsNegative}, namely $\int_M f_{\lambda_0} \dvg< 0$. For any $0<\lambda<\lambda_0$ and any $\sigma\in(-\sigma_0,0)$, where the number $\sigma_0=\sigma_0(\lambda)$ will be determined in Lemma \ref{bdoftotalQcurvature1} below, we choose $u_{0\lambda}^\sigma\in X_{f_\lambda}^*$ such that
\begin{equation}\label{initialenergy}
 \mathscr{E}(u_{0\lambda}^\sigma)\leqslant\beta_\lambda+\sigma^2,
\end{equation}
where, as in \eqref{betalambda}, we set
\[
\beta_\lambda = \min\big\{\mathscr{E}(u):u\in X^*_{f_\lambda}\big\}.
\]
For such an initial data $u_{0\lambda}^\sigma$, it follows from Theorem \ref{NgoZh} that the flow \eqref{eeforu} possesses the smooth solution
\[
u_\lambda^\sigma=u_\lambda^\sigma(t)
\]
with $\alpha_\lambda^\sigma=\alpha_\lambda^\sigma(t)$. We also let $g_\lambda^\sigma=e^{2u_\lambda^\sigma}g_0$. 

First, we establish the following simple result.

\begin{lemma}\label{lemBoundIntegralOfExpAlongFlow}
For any real number $\alpha$, there exists a constant $\Cscr_B>0$ independent of $\alpha$ and time such that
\[
\int_M e^{\alpha u_\lambda^\sigma} \dvg < \Cscr_B,
\]
where $u_\lambda^\sigma$ is a solution to the flow \eqref{eeforu} with the initial data $u_{0\lambda}^\sigma$ satisfying \eqref{initialenergy}.
\end{lemma}

\begin{proof}
Observe that $u_\lambda^\sigma \in X^*_{f_\lambda}$ and
\[
\int_M e^{\alpha u_\lambda^\sigma} \dvg = e^{\alpha \overline u_\lambda^\sigma} \int_M e^{\alpha ( u_\lambda^\sigma - \overline u_\lambda^\sigma)} \dvg.
\]
Of course the case $\alpha = 0$ is trivial. If $\alpha < 0$, then as in the proof of Lemma \ref{unbdofbeta} we apply Adam's inequality \eqref{TrudingerInequality} to get
\[
e^{- 4\overline u_\lambda^\sigma} = \int_M e^{4 (u_\lambda^\sigma - \overline u_\lambda^\sigma )} \dvg \leqslant \Cscr_A \exp \Big( \frac{1}{16 \pi^2}\mathscr{E}(u_{0\lambda}^\sigma) \Big).
\]
Using this, we can bound $\int_M \exp (\alpha u_\lambda^\sigma ) \dvg$ from above as follows
\[
\int_M e^{\alpha u_\lambda^\sigma} \dvg 
\leqslant \Big( \Cscr_A \exp \Big( \frac{1}{16 \pi^2}\mathscr{E}(u_{0\lambda}^\sigma) \Big) \Big)^{|\alpha|/4} 
\Cscr_A \exp \Big( \frac{\alpha^2}{256 \pi^2}\mathscr{E}(u_{0\lambda}^\sigma) \Big).
\]
If $\alpha>0$, then as in the proof of Lemma \ref{concentration0} we know that $\overline u_\lambda^\sigma \leqslant 0$, which then implies that
\[
\int_M e^{\alpha u_\lambda^\sigma} \dvg 
\leqslant \int_M e^{\alpha ( u_\lambda^\sigma - \overline u_\lambda^\sigma)} \dvg
\leqslant \Cscr_A \exp \Big( \frac{\alpha^2}{256 \pi^2}\mathscr{E}(u_{0\lambda}^\sigma) \Big).
\]
Putting these estimates together we obtain the existence of $\Cscr_B$. Clearly, $\Cscr_B$ is independent of $\alpha$ and time, however, $\Cscr_B$ depends on $\sigma_0$ and $\lambda$.
\end{proof}

The following lemma is the key result of this section.

\begin{lemma}\label{bdoftotalQcurvature1}
 There holds
 \begin{align*}
\liminf_{\lambda\searrow0}\limsup_{\sigma\nearrow0}& \limsup_{t \to +\infty}\int_M|Q_{g_\lambda^\sigma}| \dv_{g_\lambda^\sigma}\\
 & \leqslant2\liminf_{\lambda\searrow0}\limsup_{\sigma\nearrow0}\limsup_{t \to +\infty}(\lambda\alpha_\lambda^\sigma(t))\leqslant2\liminf_{\lambda\searrow0}(\lambda|\beta'_\lambda|)\leqslant128\pi^2.
 \end{align*}
\end{lemma}

\begin{proof}
We split our proof into two steps as follows.

\medskip
\noindent\textbf{Step 1}. We claim that for any $0<\lambda<\lambda_0$ we can find some $\sigma_0=\sigma_0(\lambda)>0$ sufficiently small such that for each $t\geqslant0$ and $\sigma\in(-\sigma_0,0)$ we have
\[
u_\lambda^\sigma+\sigma f_\lambda\in X_{f_\mu}
\]
for some $\mu=\mu(t)>\lambda$ with
\[
C^{-1}|\sigma|\leqslant|\mu-\lambda|\leqslant C|\sigma|,
 \]
where $C>0$ is constant independent of $t$ and $\sigma$. To see this, we notice from \eqref{FlowRemainInSpace} that $u_\lambda^\sigma(t)\in X^*_{f_\lambda}$ for all $t\geqslant0$. By mean value theorem, there exists two functions $\sigma', \sigma''$ valued in $(\sigma,0)$ such that
 \begin{align*}
\int_Mf_\lambda e^{4(u_\lambda^\sigma+\sigma f_\lambda)} \dvg &= \int_Mf_\lambda \Big[e^{4(u_\lambda^\sigma+\sigma f_\lambda)}-e^{4u_\lambda^\sigma}\Big] \dvg \\
&= 4\sigma\int_Mf_\lambda^2 e^{4(u_\lambda^\sigma+\sigma' f_\lambda)} \dvg 
 \end{align*}
 and
 \begin{align*}
\int_Me^{4(u_\lambda^\sigma+\sigma f_\lambda)} \dvg &= 1+\int_M\Big[e^{4(u_\lambda^\sigma+\sigma f_\lambda)}-e^{4u_\lambda^\sigma}\Big] \dvg \\
&= 1+4\sigma\int_Mf_\lambda e^{4u_\lambda^\sigma} \dvg +8\sigma^2\int_Mf_\lambda^2 e^{4(u_\lambda^\sigma+\sigma'' f_\lambda)} \dvg \\
&= 1+8\sigma^2\int_Mf_\lambda^2 e^{4(u_\lambda^\sigma+\sigma'' f_\lambda)} \dvg .
 \end{align*}
Therefore, in view of the identity
\begin{align*}
\int_M f_\mu e^{4(u_\lambda^\sigma+\sigma f_\lambda)} \dvg = \int_M f_\lambda e^{4(u_\lambda^\sigma+\sigma f_\lambda)} \dvg + (\mu -\lambda) \int_M e^{4(u_\lambda^\sigma+\sigma f_\lambda)} \dvg,
\end{align*}
if we let $\mu$ be
\begin{equation} \label{eqformu}
 \mu=\lambda-\frac{4\sigma\int_Mf_\lambda^2 e^{4(u_\lambda^\sigma+\sigma' f_\lambda)} \dvg }{1+8\sigma^2\int_Mf_\lambda^2 e^{4(u_\lambda^\sigma+\sigma'' f_\lambda)} \dvg }>\lambda,
\end{equation}
depending on $t$, then $u_\lambda^\sigma+\sigma f_\lambda\in X_{f_\mu}$. To bound $|\mu - \lambda|$, we need further estimates for numerator and denominator of $\mu$ in \eqref{eqformu}. First we note by H\"older's inequality that
\[
\int_Mf_\lambda^2 e^{4(u_\lambda^\sigma+\sigma' f_\lambda)} \dvg 
\geqslant \Big( \int_Mf_\lambda \dvg \Big)^2 \Big( \int_M e^{-4(u_\lambda^\sigma+\sigma' f_\lambda)} \dvg \Big)^{-1}.
\]
Because
\begin{align*}
\int_M e^{-4(u_\lambda^\sigma+\sigma' f_\lambda)} \dvg &= \Big( \int_{\{ f_\lambda \leqslant 0 \}} + \int_{\{ f_\lambda > 0 \}} \Big) e^{-4(u_\lambda^\sigma+\sigma' f_\lambda)} \dvg\\
& \leqslant \exp \big( 4|\sigma| \|f_\lambda\|_{L^\infty(M, g_0)} \big) \int_{\{ f_\lambda \leqslant 0 \}} e^{-4u_\lambda^\sigma} \dvg + \int_{\{ f_\lambda > 0 \}} e^{-4u_\lambda^\sigma},
\end{align*}
which implies that
\[
\int_M e^{-4(u_\lambda^\sigma+\sigma' f_\lambda)} \dvg \leqslant 2 \int_M e^{-4u_\lambda^\sigma} \dvg
\]
if we further choose $|\sigma|$ sufficiently small. From this and Lemma \ref{lemBoundIntegralOfExpAlongFlow} we deduce that there exists some positive constant $c(\lambda,\sigma_0)$ depending only on $\lambda$ and $\sigma_0$ such that
\begin{equation}\label{eqLOWERBOUND}
\int_Mf_\lambda^2 e^{4(u_\lambda^\sigma+\sigma' f_\lambda)} \dvg \geqslant c(\lambda,\sigma_0)
\end{equation}
for any $0<\lambda<\lambda_0$, any $\sigma \in (-\sigma_0, 0)$, and any $\sigma' \in (\sigma,0)$. Moreover, it is easy to see that
\begin{equation}\label{eqUPPERBOUND}
\begin{aligned}
\int_Mf_\lambda^2 e^{4(u_\lambda^\sigma+\sigma' f_\lambda)} \dvg & \leqslant \| f_\lambda \| ^2_{L^\infty(M, g_0)} \exp \big( 4|\sigma| \| f_\lambda \| _{L^\infty(M, g_0)} \big) \int_Me^{4u_\lambda^\sigma} \dvg \\
&\leqslant \| f_\lambda \| ^2_{L^\infty(M, g_0)} \exp \big( 4|\sigma_0| \| f_\lambda \| _{L^\infty(M, g_0)} \big)
\end{aligned}
\end{equation}
 for any $t\geqslant0$ and any $\sigma' \in(\sigma,0)$. From this we can bound $|\mu - \lambda|$ from above as follows
\[\begin{aligned}
|\mu - \lambda| & \leqslant 4 |\sigma| \int_Mf_\lambda^2 e^{4(u_\lambda^\sigma+\sigma' f_\lambda)} \dvg\\
& \leqslant 4 \| f_\lambda \| ^2_{L^\infty(M, g_0)} \exp \big( 4|\sigma_0| \| f_\lambda \| _{L^\infty(M, g_0)} \big) |\sigma| .
\end{aligned}\]
Moreover, we can also bound $|\mu - \lambda|$ from below, thanks to \eqref{eqLOWERBOUND} and \eqref{eqUPPERBOUND}. Hence, for any $0<\lambda<\lambda_0$ we can find $\sigma_0=\sigma_0(\lambda)>0$ such that
 \[
C(\lambda)^{-1}|\sigma|\leqslant|\mu-\lambda|\leqslant C(\lambda)|\sigma|
\]
for all $\sigma\in(-\sigma_0,0)$, where $C(\lambda)>0$ is independent of $t\geqslant0$ and $\sigma$ but could depend on $\lambda$. The claim is thus proved.
 \smallskip
 
\medskip
\noindent\textbf{Step 2}. It follows from \cite[Lemmas 4.1 and 6.3 ]{NZ} that $\alpha_\lambda^\sigma(t)$ and $u_\lambda^\sigma(t)$ are uniformly bounded in time $t$ and $\sigma$. Notice that by the relations
\[
Q_{g_\lambda^\sigma} e^{4u_\lambda^\sigma}= \Po u_\lambda^\sigma, \quad u_{\lambda,t}^\sigma=\alpha_\lambda^\sigma f_\lambda-Q_{g_\lambda^\sigma}
\]
we can expand $\mathscr{E}(u_\lambda^\sigma+\sigma f_\lambda)$ to get
 \begin{align*}
 \mathscr{E}(u_\lambda^\sigma+\sigma f_\lambda)&= \mathscr{E}(u_\lambda^\sigma)+4\sigma\int_M \Po u_\lambda^\sigma f_\lambda \dvg +\sigma^2\mathscr{E}(f_\lambda)\\
 &= \mathscr{E}(u_\lambda^\sigma)+4\sigma\alpha_\lambda^\sigma\int_M f_\lambda^2e^{4u_\lambda^\sigma} \dvg -4\sigma\int_M u_{\lambda,t}^\sigma f_\lambda e^{4u_\lambda^\sigma} \dvg +\sigma^2\mathscr{E}(f_0).
 \end{align*}
 Observing that \cite[Lemma 6.1]{NZ} yields
 \begin{equation}\label{l2convergence}
 \int_M|u_{\lambda,t}^\sigma|^2e^{4u_\lambda^\sigma} \dvg =\int_M|\alpha_\lambda^\sigma f_\lambda-Q_\lambda^\sigma|^2e^{4u_\lambda^\sigma} \dvg \to 0
 \end{equation}
 as $t \to +\infty.$ Hence, by \eqref{volumekeeping} and H\"older's inequality, we can estimate
 \[
\Big|\int_Mu_{\lambda,t}^\sigma f_\lambda e^{4u_\lambda^\sigma} \dvg \Big|\leqslant \| f_\lambda \| _{L^\infty (M,g_0)} \Big(\int_M|u_{\lambda,t}^\sigma|^2e^{4u_\lambda^\sigma} \dvg \Big)^{1/2} \to 0
\]
 as $t \to +\infty$.
 Since the energy $\mathscr{E}(u_\lambda^\sigma)$ is decay along the flow, we have, by \eqref{initialenergy} and the expansion of $\mathscr{E}(u_\lambda^\sigma+\sigma f_\lambda)$ above, that
 \begin{align*}
 \beta_\mu&\leqslant \mathscr{E}(u_\lambda^\sigma+\sigma f_\lambda ) \\
&\leqslant\mathscr{E}(u_\lambda^\sigma) +4\sigma\alpha_\lambda^\sigma\int_M f_\lambda^2e^{4u_\lambda^\sigma} \dvg +\sigma^2\mathscr{E}(f_0)+o(1)\\
 &\leqslant \beta_\lambda+4\sigma\alpha_\lambda^\sigma\int_M f_\lambda^2e^{4u_\lambda^\sigma} \dvg +\sigma^2(1+\mathscr{E}(f_0))+o(1).
 \end{align*}
 However, from \eqref{eqformu} we obtain
 \[
4\sigma\int_M f_\lambda^2e^{4u_\lambda^\sigma} \dvg =\lambda-\mu+4\sigma I,
\]
with 
\[
I=\int_Mf_\lambda^2 h e^{4u_\lambda^\sigma} \dvg,
\]
where 
\[
h = 1-\frac{e^{4\sigma' f_\lambda}}{1+8\sigma^2\int_Mf_\lambda^2 e^{4(u_\lambda^\sigma+\sigma'' f_\lambda)} \dvg }.
\]
Clearly
\[
|h| \leqslant 8\sigma^2\int_Mf_\lambda^2 e^{4(u_\lambda^\sigma+\sigma'' f_\lambda)} \dvg + \big| 1 - e^{4\sigma' f_\lambda} \big|.
\]
Because $\sigma' \in (\sigma, 0)$, there is some constant $C>0$ independent of $t$ and $\sigma$ such that
\[
\|h \|_{L^\infty(M, g_0)} \leqslant C |\sigma|.
\]
Keep in mind that $u_\lambda^\sigma \in X^*_{f_\lambda}$. From this we can use \eqref{volumekeeping} to bound $I$ as follows
\[
|I|\leqslant \|f_\lambda \|_{L^\infty(M, g_0)}^2 \|h \|_{L^\infty(M, g_0)} \leqslant C(\lambda)|\mu-\lambda|,
\]
where $C(\lambda)>0$ is a uniform constant independent of $t$ and $\sigma$. Therefore, with error $o(1) \to 0$ as $t \to +\infty$ and the uniform bound of $\alpha_\lambda^\sigma$ in $t$ and in $\sigma$ we arrive at the estimate
\begin{align*}
\beta_\mu &\leqslant \beta_\lambda+ \alpha_\lambda^\sigma ( \lambda-\mu+4\sigma I )+\sigma^2(1+\mathscr{E}(f_0))+o(1) \\
&= \beta_\lambda-\alpha_\lambda^\sigma(\mu-\lambda)+O(1) (\mu-\lambda)^2+o(1),
\end{align*}
where $O(1)$ is independent of $t$ but could depend on $\lambda$ and $\sigma_0$. This implies that
\[
\limsup_{t \to +\infty}\alpha_\lambda^\sigma(t)\leqslant\limsup_{t \to +\infty} \Big( \frac{\beta_\lambda-\beta_\mu}{\mu-\lambda}+O(1)(\mu-\lambda) \Big).
\]
Now, as $\sigma\nearrow0$, we have from \eqref{eqformu} that $\mu\searrow\lambda$ uniformly in time $t>0$. So, for almost every $\lambda\in(0,\lambda_0)$ there holds
\[
\limsup_{\sigma\nearrow0}\limsup_{t \to +\infty}\alpha_\lambda^\sigma(t)\leqslant\lim_{\mu\searrow\lambda}\frac{\beta_\lambda-\beta_\mu}{\mu-\lambda}=|\beta_\lambda'|.
\]
Multiplying both sides by $\lambda>0$, as in the proof of Lemma \ref{bdoftotalQcurvature}, we find that
\[
\liminf_{\lambda\searrow0}\limsup_{\sigma\nearrow0}\limsup_{t \to +\infty}(\lambda\alpha_\lambda^\sigma)\leqslant\liminf_{\lambda\searrow0}(\lambda|\beta_\lambda'|)\leqslant64\pi^2.
\]
Finally, it follows from the flow equation \eqref{eeforu} and \eqref{bdoff_lambda} that
\[
|Q_{g_\lambda^\sigma}| \leqslant \alpha_\lambda^\sigma |f_\lambda| +u_{\lambda,t}^\sigma \leqslant 2 \lambda \alpha_\lambda^\sigma - \alpha_\lambda^\sigma f_\lambda +u_{\lambda,t}^\sigma,
\]
which then gives
\[
\int_M|Q_{g_\lambda^\sigma}| \dv_{g_\lambda^\sigma}
\leqslant 2\lambda \alpha_\lambda^\sigma +\int_M|u_{\lambda,t}^\sigma| \dv_{g_\lambda^\sigma} 
= 2\lambda\alpha_\lambda^\sigma+o(1),
\]
thanks to \eqref{volumekeeping} and \eqref{l2convergence}. From this the lemma follows.
\end{proof}

\subsection{Bubbling of the prescribed curvature flow}
In this subsection, we devote ourselves to prove the blow-up behavior along the prescribed $Q$-curvature flow, namely Theorem \ref{main2}. 
From Lemma \ref{bdoftotalQcurvature1}, it follows that there exists a sequence $\lambda_k\searrow0$ such that
$$\sup_{k\in\mathbb{N}}\limsup_{\sigma\nearrow0}\limsup_{t \to +\infty}(\lambda_k\alpha_{\lambda_k}^\sigma(t)-1/k)\leqslant64\pi^2.$$
We may then fix a sequence $\sigma_k\nearrow0$ with
$$\sup_{k\in\mathbb{N}}\sup_{\sigma_k\leqslant\sigma<0}\limsup_{t \to +\infty}(\lambda_k\alpha_{\lambda_k}^\sigma(t)-2/k)\leqslant64\pi^2.$$
Choosing $\sigma=\sigma_k$ for each $k\in\mathbb{N}$, we find, for suitable $T_k \to +\infty$ satisfying
\[
F_k(t):=\int_M|u_{\lambda_k,t}^{\sigma_k}(t)| \dv_{g_{\lambda_k}^{\sigma_k}}\leqslant1/k
\]
for $t_k\geqslant T_k$ , we find the bound
\begin{equation}
 \label{sequentialbdoftotalQcurvature}
 \sup_{t\geqslant T_k}\int_M|Q_{g_{\lambda_k}^{\sigma_k}}| \dv_{g_{\lambda_k}^{\sigma_k}}\leqslant\sup_{t\geqslant T_k}\big(2\lambda_k\alpha_{\lambda_k}^{\sigma_k}+F_k(t)\big)\leqslant128\pi^2+5/k,
\end{equation}
 for any $k\in\mathbb{N}$. Hence, if for each $k\in\mathbb{N}$ for any $t_k\geqslant T_k$ we let $w_k=u_{\lambda_k}^{\sigma_k}(t_k)$, then $w_k$ satisfies \eqref{pQeerror} with $\alpha_k=\alpha_{\lambda_k}^{\sigma_k}(t_k)$ and $h_k=u_{\lambda_k,t}^{\sigma_k}(t_k)$. From this we can apply Theorems \ref{main3} and \ref{main4} to get the desired result. This completes the proof of Theorem \ref{main2}.

\section*{Acknowledgments}

Q.A.N. would like to thank Professor Michael Struwe for sharing his preprint \cite{Str} at an early stage and for many fruitful discussions regarding curvature flows. Part of this work was done while Q.A.N was visiting the Vietnam Institute for Advanced Study in Mathematics (VIASM) in 2018, he gratefully acknowledges the institute for hospitality and support. He also acknowledges the support from the Vietnam National University, Hanoi (VNU) under project number QG.19.12. The research of H.Z. is supported by the Young Scientists Fund of the National Natural Science Foundation of China (Grant No. 11701544). Thanks also go to Luca Galimberti for his careful reading of the paper and a few suggestions leading to the present version. Last but not least, both authors would like to thank Professor Xingwang Xu for his interests and constant encouragement over the years.


\end{document}